\documentclass[11pt,a4paper]{article}
\usepackage{enumerate}
\usepackage{amsmath}
\usepackage{fix-cm}
\usepackage{pb-diagram}
\usepackage[english]{babel}
\usepackage{amssymb,latexsym}
\usepackage{url}
\usepackage{atbegshi}
\usepackage{enumitem}  
\usepackage{float}
\usepackage{tikz-cd}
\usepackage{mathalpha}
\usepackage[dvipsnames]{xcolor}
\usepackage{authblk}
\usepackage{mathrsfs}
\usepackage{bm}
\usepackage{caption}

\newcommand{\thref}[1]{\hyperref[#1]{Theorem~\ref*{#1}}}
\newcommand{\corref}[1]{\hyperref[#1]{Corollary~\ref*{#1}}}
\newcommand{\propref}[1]{\hyperref[#1]{Proposition~\ref*{#1}}}
\newcommand{\secref}[1]{\hyperref[#1]{Section~\ref*{#1}}}
\newcommand{\lemref}[1]{\hyperref[#1]{Lemma~\ref*{#1}}}
\newcommand{\rkref}[1]{\hyperref[#1]{Remark~\ref*{#1}}}
\newcommand{\defref}[1]{\hyperref[#1]{Definition~\ref*{#1}}}
\newcommand{\notref}[1]{\hyperref[#1]{Notation~\ref*{#1}}}
\newcommand{\Conref}[1]{\hyperref[#1]{Conjecture~\ref*{#1}}}
\usepackage{amsthm}
\usepackage[pdftex, bookmarks, colorlinks=true, pdfpagemode=UseOutlines,
        linkcolor=blue,
        urlcolor=primaryblue,
        ]{hyperref}
\usepackage{hyperref}

\definecolor{primaryblue}{RGB}{70,130,180}

\usepackage{comment}

\DeclareMathOperator{\radd}{rad}
\DeclareMathOperator{\tann}{tan}
\DeclareMathOperator{\Abs}{{Abs,p}}

\DeclareMathOperator{\spec}{Spec}

\DeclareMathOperator{\im}{Im}
\DeclareMathOperator{\Coex}{Coex}
\DeclareMathOperator{\Ex}{Ex}

\DeclareMathOperator{\w}{\omega}

\def \bui#1#2{\mathrel{\mathop{\kern 0pt#1}\limits^{#2}}}
\def \buil#1#2{\mathrel{\mathop{\kern 0pt#1}\limits_{#2}}}

\let\<\langle
\let\>\rangle
\makeatletter
\renewcommand{\@cite}[2]{%
  \textcolor{green}{[}
  {#1\if@tempswa \textcolor{green}{, #2}\fi}
  \textcolor{green}{]}
}
\makeatother

\newtheorem{example}{Examples}[section]
\newtheorem{thm}{Theorem}[section]
\newtheorem{lemma}[thm]{Lemma}
\newtheorem{prop}[thm]{Proposition}

\newtheorem{remark}[thm]{Remark}
\newtheorem{remarks}[thm]{Remarks}
\newtheorem{definition}[thm]{Definition}
\newtheorem{notation}[thm]{Notation}
\newtheorem{exabout:ample}[thm]{Example}

\begin{document}
\date{\normalsize \today}

\title{On the spectrum of the biharmonic Steklov operator with Neumann boundary conditions acting on forms of the unit Euclidean ball}
\author[1,2]{Rodolphe Abou Assali\thanks{\texttt{rodolphe.abou-assali@univ-lorraine.fr, rodolphe.abouassali.1@ul.edu.lb}}}

\affil[1]{{\footnotesize Lebanese University, Faculty of Sciences II, Department of Mathematics, P.O. Box 90656 Fanar-Matn, Lebanon}}
\affil[2]{\footnotesize Universit\'e de Lorraine, CNRS, IECL, F-54000 Nancy, France}

\newcommand{\myAbst}[1]{
    \begin{center}
        \begin{minipage}{0.7\textwidth}
        \textbf{Abstract}. {\small #1}
        \end{minipage}
    \end{center}
}

\maketitle 
\myAbst{%
We compute the spectrum of the biharmonic Steklov operator with Neumann boundary conditions acting on differential $p$-forms over the unit Euclidean ball. We apply it to test how far one of Kuttler-Sigillito inequalities is from being optimal.
}


\section{Introduction} \label{sec:introduction}
Let $(M^n,g)$ be a smooth compact manifold of dimension $n\geq2$ with smooth boundary. Let us start by recalling the absolute de Rham cohomology of degree $p$ is the finite dimensional space \begin{equation}\label{eq:Hap}
        H_{\Abs}(M) = \left\{ \w \in \Omega^p(M) \mid d\w = 0 \text{ on } M, \ \delta\w = 0 \text{ on } M, \ \nu \lrcorner \w = 0 \text{ on } \partial M \right\}.
 \end{equation} We refer to \cite[Section 2]{RaulotSavo} for more details. We recall the Steklov (Dirichlet-to-Neumann) problem for differential forms: \cite[Section 1.1]{RaulotSavo}
\begin{equation}\label{eq:steklovFormes}
    \left\{
    \hspace{-0.3em}
    \begin{array}{r@{\quad}c@{\quad}l@{\quad}l}
        \Delta {\w} &=& 0 & \text{on } M \\[0.15em]
        \nu \lrcorner {\w} &=& 0 & \text{on } \partial M \\[0.15em]
        \nu \lrcorner d\w &=& -\sigma \iota^* {\w} & \text{on } \partial M
    \end{array}
    \right.
\end{equation}where $\sigma\in \mathbb{R}$. 
 The spectrum of the Steklov operator on forms is discrete, consisting of eigenvalues 
\[0 < \sigma_{1,p} \leq \sigma_{2,p} \leq \ldots 
\] Note that $\sigma_{0,p}=0$ is an eigenvalue if and only if  $H_{\Abs}(M)\neq0.$ For $p\in\{0,\cdots,n\}$, the kernel of the Steklov problem is $H_{\Abs}(M)$. 
We review the scalar biharmonic Steklov eigenvalue problem with Dirichlet boundary conditions: 
\begin{equation}\label{eq:DBSfct}
    \left\{
    \hspace{-0.3em}
    \begin{array}{r@{\quad}c@{\quad}l@{\quad}l}
        \Delta^2 f &=& 0 & \text{on } M \\[0.15em]
        f &=& 0 & \text{on } \partial M \\[0.15em]
        \Delta f - q \partial_\nu f &=& 0 & \text{on } \partial M
    \end{array}
    \right.
\end{equation}where $f$ is a smooth function on $M$ and the real number $q$ is the eigenvalue of $DBS^{func}$. Its spectrum is discrete, consisting of positive eigenvalues. The latter was studied by A. Ferrero, F. Gazzola and T. Weth in \cite{FerreroGazzolaWeth} where the authors gave a physical interpretation of the first eigenvalue and of the Steklov boundary conditions. A generalization of \eqref{eq:DBSfct} to the setting of differential forms and is noted by DBS and given by:
\begin{thm}\cite[Theorem 2.2]{BSFidaGeorgeOlaNicolas}\label{thm:DBSTHM}
The boundary-value problem \begin{equation}\label{eq:DBSF1}
    (DBS)
    \left\{
    \hspace{-0.3em}
    \begin{array}{r@{\quad}c@{\quad}l@{\quad}l}
        \Delta^2 \w &=& 0 & \text{on } M \\[0.15em]
        \w &=& 0 & \text{on } \partial M \\[0.15em]
        \nu \lrcorner \Delta \w + q \iota^* \delta\w &=& 0 & \text{on } \partial M \\[0.15em]
        \iota^* \Delta \w - q\nu \lrcorner d\w &=& 0 & \text{on } \partial M
    \end{array}
    \right.
\end{equation} on $p$-forms, has a discrete spectrum consisting of an unbounded monotonously non-decreasing sequence of positive eigenvalues of finite multiplicities $(q_{j,p})_{j\geq1}.$\end{thm} 

We now state the scalar biharmonic Steklov eigenvalue problem with Neumann boundary conditions:
\begin{equation}\label{eq:IntroNBS}
    \left\{
    \hspace{-0.3em}
    \begin{array}{r@{\quad}c@{\quad}l@{\quad}l}
        \Delta^2 f &=& 0 & \text{on } M \\[0.15em]
        \partial_\nu f &=& 0 & \text{on } \partial M \\[0.15em]
        \partial_\nu (\Delta f) - \ell f &=& 0 & \text{on } \partial M
    \end{array}
    \right.
\end{equation}
where $f$ is a smooth function on $M$ and the real number $\ell$ is the eigenvalue of $NBS^{func}$. Its spectrum is discrete, consisting of non-negative eigenvalues. The lowest one is always zero, corresponding to a constant eigenfunction.  A generalization of \eqref{eq:IntroNBS} to the settings of $p$-differential forms is defined in \cite{Article1} and given by: \begin{thm}\label{thm:TheoremeGeneralNBSF3PourLaPreuve} Let $(M^n,g)$ be a compact Riemannian manifold with a smooth boundary $\partial M$ and let $\nu$ be the unit inward normal vector to the boundary. Then the following boundary-value problem: \begin{equation}\label{eq:NBSF3preuve}
    (NBS)
    \left\{
    \hspace{-0.3em}
    \begin{array}{r@{\quad}c@{\quad}l@{\quad}l}
        \Delta^2 \w &=& 0 & \text{on } M \\[0.15em]
        \nu \lrcorner \w &=& 0 & \text{on } \partial M \\[0.15em]
        \nu \lrcorner d\w &=& 0 & \text{on } \partial M \\[0.15em]
        \nu \lrcorner \Delta \w &=& 0 & \text{on } \partial M \\[0.15em]
        \nu \lrcorner d\Delta \w + \ell \iota^* \w &=& 0 & \text{on } \partial M
    \end{array}
    \right.
\end{equation} on $p$-forms, has a discrete spectrum consisting of an unbounded non-decreasing sequence of positive real eigenvalues with finite multiplicities $(\ell_{i,p})_{i\geq0}$. Its kernel is given by $H_{\Abs}(M)$.
\end{thm} In two recent works \cite{Article1, Article2}, we generalized some of Kuttler-Sigillito inequalities to the setting of differential forms. These inequalities basically link the eigenvalues of different boundary-value problems. Initially established for bounded domains in $\mathbb{R}^2$ by J.R. Kuttler and V.G Sigillito \cite{Kuttler&Sigillito}, some of them were later generalized to the setting of compact Riemannian manifolds \cite{Berge, hassannezhadETsiffert}.\vspace*{\baselineskip}

The aim of this paper is to compute the spectrum of the NBS operator \eqref{eq:NBSF3preuve} for differential forms on the Euclidean ball. We obtain the following: \begin{thm}\label{thm:spectreNBS}
  Let us fix $n \ge 3$ and $1 \le p \le n-2$. The spectrum of the NBS operator on the $p$-forms of the unit ball of $\mathbb{R}^n$ is given by: \[\{\ell_{k,p}^{\Ex}\mid k\in\mathbb{N}\}\cup \{\ell_{k,p}^{\Coex}\mid k\in\mathbb{N}\} \] where \begin{equation}\label{eq:1erspectre}
         \ell^{\Coex}_{k,p}=(k+p+1)^2(n+2(k+1))
     \end{equation} corresponds to a NBS-eigenform $\omega_{k,p}^{\Coex}$ such that $\iota^* \omega_{k,p}^{\Coex}$ is a co-exact eigenform of the Laplacian on $\mathbb{S}^{n-1}$ associated to $\boldsymbol{\lambda^{\Coex}_{k,p}}=(k+p+1)(k+n-p-1)$ and
\begin{equation*}
    \ell^{\Ex}_{k,p}=\frac{(2k+n+2)^2(k+p)^2(2k+n+4)}{4k^2+4kn+4k+n^2+4n-4p}
\end{equation*} corresponds to a NBS-eigenform $\omega_{k,p}^{\Ex}$ such that $\iota^* \omega_{k,p}^{\Ex}$ is an exact eigenform of the Laplacian on $\mathbb{S}^{n-1}$ associated to $\boldsymbol{\lambda^{\Ex}_{k,p}}=(k+p)(k+n-p)$. Moreover, the multiplicity of these eigenvalues is given by: \begin{align*}
\dim(\ker(\mathrm{NBS}^{\mathrm{form}} - \ell^{\mathrm{Coex}}_{k,p})) 
&= \dim(\ker(\bar{\Delta}_p - \boldsymbol{\lambda^{\mathrm{Coex}}_{k,p}})) \\
&=\frac{(n + k-3)! (n + 2k-2)}{p!  k! (n - p - 4)!  (n + k - p-3)(k + p + 1)}
\end{align*}

\noindent and

\begin{align*}
\dim(\ker(\mathrm{NBS}^{\mathrm{form}} - \ell^{\mathrm{Ex}}_{k,p})) 
&= \dim(\ker(\bar{\Delta}_p - \boldsymbol{\lambda^{\mathrm{Ex}}_{k,p}})) \\
&= \frac{(n + k-3)! (n + 2k-2)}{(p - 1)!  k!  (n - p-3)!  (n + k - p - 2 )(k + p)}.
\end{align*}\end{thm}
The proof of this theorem relies on several intermediate results, see \secref{sec:sectionNBSSpectre}.
Note that for $ p = 0,$ the expression \eqref{eq:1erspectre} coincides with the one found by Q. Wang and C. Xia in \cite[Theorem 1.5]{WangXia}, we mention this case in the following Theorem:

\begin{thm}\label{thm:spectreNBScasparticulier}
    Let us fix $n \ge 2$.
\begin{itemize}
    \item If $p = 0$, the spectrum of NBS on $0$-forms (i.e., scalar functions) on the unit ball of $\mathbb{R}^n$ is given by
    \[
    \{\ell_{k,0} \mid k \in \mathbb{N}\} = \{0\} \cup \{\ell_{k,0}^{\Coex} \mid k \in \mathbb{N}\}=\{k^2(n+2k)\mid k\in\mathbb{N}\}.
    \]
    The eigenvalue $\ell_{k,0}$ is associated with a NBS eigenfunction $f_{k,0}$ such that ${(f_{k,0})}_{|\mathbb{S}^{n-1}}$ is an eigenfunction of the Laplacian on $\mathbb{S}^{n-1}$ with eigenvalue $\boldsymbol{\lambda_{k,0}}=k(k+n-2)$. Its multiplicity is given by
    \[
    \dim(\ker(NBS^{\text{func}}- \ell_{k,0})) = \dim(\ker(\bar{\Delta}_0 - \boldsymbol{\lambda_k})).
    \]
    \item If $p = n-1$, the spectrum of NBS on $(n-1)$-forms on the unit ball of $\mathbb{R}^n$ is given by \begin{multline*}
\{n^3(n+2)\} \cup \{\ell^{\mathrm{Ex}}_{k,n-1} \mid k \in \mathbb{N}\} \\
= \{n^3(n+2)\} \cup \left\{ \frac{(2k+n+2)^2(k+n-1)^2(2k+n+4)}{4k^2+4kn+4k+n^2+4} \;\middle|\; k \in \mathbb{N} \right\}
\end{multline*}where $n^3(n+2)$ has multiplicity $1$ and is associated to a NBS-eigenform $\bar{\omega}$ such that $\iota^*\bar{\omega}$ is proportional to the volume form of $\mathbb{S}^{n-1}$ and $\ell^{\Ex}_{k,n-1}$ corresponds to a NBS eigenform $\w_{k,n-1}$ such that $\iota^*\w_{k,n-1}$ is an exact eigenform of the Laplacian on $\mathbb{S}^{n-1}$ associated to $\boldsymbol{\lambda^{\Ex}_{k,n-1}}=\boldsymbol{\lambda_{k+1,0}}=(k+1)(k+n-1)$ and has multiplicity
    \[
    \dim(\ker(NBS^{\text{form}} - \ell^{\Ex}_{k,n-1})) = \dim(\ker(\bar{\Delta}_{n-1} - \boldsymbol{\lambda^{\Ex}_{k,n-1}})).
    \]
\end{itemize}
\end{thm}
 Afterwards, we test how far the following Kuttler-Sigillito inequality is from being optimal, we recall:

\begin{thm}\cite{Article1}\label{thm:proofQSigma<l}
    Let $(M^n, g)$ be a compact Riemannian manifold with smooth boundary $\partial M$. The following inequality holds between the eigenvalues of the Steklov, DBS, and NBS problems for differential forms:
    \begin{equation}\label{eq:q1sigma1<l1PourLaPreuve}
    q_{1,p} \sigma_{1,p}^2 < \ell_{1,p}.
    \end{equation}
\end{thm}
This last theorem generalizes \cite[Inequality II]{Kuttler&Sigillito}.
We get the following result: \begin{thm}
For $n\geq3$ and $p\in\{1,\cdots,n-2\}$, we have the following relation on the Euclidean ball:
\begin{equation*}
\frac{q_{1,p} \sigma_{1,p}^2}{ \ell_{1,p}} = 
\begin{cases} 
\dfrac{n^2+4n-4p}{n(n+4)} & \text{if } 1\leq p \leq \frac{n}{2} \\[2ex]
\dfrac{n}{n+2} & \text{if } p > \frac{n}{2}.
\end{cases}
\end{equation*} 
\end{thm} Therefore, we recover $q_{1,p}\sigma_{1,p}^2<\ell_{1,p}.$ See \secref{sec:KSTEST} for the proof.

\textbf{Outline of the Article:} In \secref{sec:introduction} of this paper, we recall some well-known boundary value problems and give the main results of this work. In \secref{sec:Spectre}, we review the spectrum of the $p$-Laplacian on the sphere and determine the spectrum of the NBS operator on the Euclidean ball, which is the main result of the paper. Finally, in \secref{sec:Test}, we remind the spectra of Steklov and DBS operators on the Euclidean ball and then test how far a Kuttler-Sigillito inequality is from being optimal.

\textbf{Acknowledgment:} I would like to thank my PhD supervisors, Nicolas Ginoux, Georges Habib and Samuel Tapie, for their help, support, and the valuable discussions we have had. I also acknowledge the support of the Grant ANR-24-CE40-0702 ORBISCAR.

\textit{This work was done and written independently, under the direction of my thesis supervisors. No generative artificial intelligence was used to generate results or provide ideas, with the exception of solving the system \eqref{eq:systemenonmatrice} and improving the phrasing and syntax in a few places. I claim full ownership of the work and proofs presented in this document, and I assume full scientific responsibility for them.}

\section{Spectral geometry of the Euclidean ball}\label{sec:Spectre}
In this section, we recall some classical results on the spectra of the Laplacian on the Euclidean ball. Later, we explicitly determine the spectrum of the NBS operator. Our interest in computing this spectrum, as mentioned in the introduction, lies in testing the Kuttler-Sigillito inequalities in this specific geometry. An introductory survey on the classical case where the Laplacian acts on scalar functions (0-forms) can be found in \cite{Chavel, VitoCruz}. Following \cite[p.185]{NagyPaul}, we list in the table below the few Riemannian manifolds $(M^n, g)$ without boundary for which the spectrum of the Laplacian on $p$-forms with $1 \leq p \leq n-1$ is explicitly known: 
\begin{table}[H]
    \centering
    \begin{tabular}{|r|c|}
        \hline
         flat tori & \cite{BergerGauduchonMazet} \\
        \hline
         the standard sphere & \cite{GallotMeyer} \\
        \hline
         the canonical complex projective space & \cite{AkiraYoshiharu} \\
        \hline
         compact Lie groups with a bi-invariant metric & \cite{Wallach} \\
        \hline
        compact semisimple Lie groups & \cite{BeersMillman} \\
        \hline
         Heisenberg manifolds, for $\mathbb{S}^1$-invariant $1$-forms & \cite{Heisenberg} \\
        \hline
    \end{tabular}
    \caption*{Table: Known spectra on the sphere}
\end{table} 

\subsection{Spectrum of the \texorpdfstring{$p$}{p}-Laplace operator on the sphere}\label{sec:RappelSpectre}
 We denote by $\bar{\Delta}_p$ the Laplacian on $p$-forms on the sphere. When there is no ambiguity, we simply write $\bar{\Delta}$. The following Theorem is well known: \begin{thm}\label{thm:vplaplacienFCT}
    For $n \geq 1$, the spectrum of the Laplace-Beltrami operator on scalar functions is given by
\[
\{\lambda_{k,0} \mid k \in \mathbb{N}\} = \{k(k+n-2) \mid k \in \mathbb{N}\}
\]
with multiplicity \[\frac{n+2k-2}{n+k-2}\cdot\binom{n+k-2}{k}.\]

\end{thm} A proof can for instance be found in \cite[Chapter 2, p:35]{Chavel} and \cite[Chapter 3, Proposition C.I.1, Corollary C.I.3]{BergerGauduchonMazet}. By Hodge duality, the spectrum of $\bar{\Delta}$ on $(n-1)$-forms is the same, with same multiplicity. For $n\geq2$ and $p\in\{1,\cdots,n-2\},$ the analogous Theorem has been shown by S. Gallot and D. Meyer and later by I. Iwasaki and K. Katase:
\begin{thm}\label{thm:vplaplacien}\cite[Theorem, p.283]{GallotMeyer}, \cite[Theorem 6]{IwasakiKiyoshi}
       For $n\geq3$ and $p\in\{1,\cdots,n-2\}$, the eigenvalues of the Laplacian $\bar{\Delta}$ on $\mathbb{S}^{n-1}$ are given by: \[\{\boldsymbol{\lambda^{\Ex}_{k,p}}\mid k\in\mathbb{N}\}\cup \{\boldsymbol{\lambda^{\Coex}_{k,p}}\mid k\in\mathbb{N}\} \] where
\begin{equation}\label{eq:vpexacte}
\boldsymbol{\lambda^{\Ex}_{k,p}} = (k + p)(k+n-p), \quad k = 0, 1, 2, \dots
\end{equation}correspond to the exact $p$-forms and
\begin{equation}
\boldsymbol{\lambda^{\Coex}_{k',p}} = (k' + p + 1)(k'+n-p-1), \quad k' = 0, 1, 2, \dots
\end{equation} correspond to the co-exact $p$-forms.
\end{thm}
We write $\bar{d}$ and $\bar{\delta}$ for the differential and co-differential on $\mathbb{S}^{n-1}$. Let us show the following Lemma:
 \begin{lemma}
     For all $p\in\{0,\cdots,n-2\}$, $\boldsymbol{\lambda}\in\mathbb{R}$, let $V_{\boldsymbol{\lambda}}$ be the $\boldsymbol{\lambda}$-eigenspace of $\bar{\Delta}$ on $p$-forms. The map $$\bar{d}:V_{\boldsymbol{\lambda^{\Coex}_{k,p}}}\cap\im(\bar{\delta})\longrightarrow V_{\boldsymbol{\lambda^{\Ex}_{k,p+1}}}\cap\im(\bar{d})$$ is an isomorphism.
 \end{lemma} 
 \begin{proof}
 Let $V_{\boldsymbol{\lambda^{\Coex}_{k,p}}}$ be the $\boldsymbol{\lambda^{\Coex}_{k,p}}$-eigenspace of $\bar{\Delta}$. Recall that for $1\leq p\leq n-2$, the harmonic space on the sphere is trivial. For a $\boldsymbol{\lambda^{\Coex}_{k,p}}$-eigenform $\omega$ that is co-closed (i.e., $\bar{\delta}\omega = 0$ and $\bar{\Delta}\omega = \boldsymbol{\lambda^{\Coex}_{k,p}}\omega$), we have $\bar{\Delta}\omega = \bar{d}\underbrace{\bar{\delta}\omega}_{0} + \bar{\delta} \bar{d}\omega$. It follows that $\boldsymbol{\lambda^{\Coex}_{k,p}}\bar{d}\omega = \bar{d}\bar{\delta}\omega + \bar{\delta} \underbrace{\bar{d}\bar{d}}_{0}\omega = \bar{\Delta} \bar{d}\omega$, hence, $\bar{d}\omega$ is a $\boldsymbol{\lambda^{\Coex}_{k,p}}$-eigenform. To show that $\bar{d}$ is injective, we let $\omega\in V_{\boldsymbol{\lambda^{\Coex}_{k,p}}}\cap\im(\bar{\delta})$ such that $\bar{d}\omega = 0$. We have $\w=\bar{\delta}\alpha$ for some $\alpha\in\Omega^{p-1}(\mathbb{S}^{n-1}),$ therefore $(\w,\w)_{L^2(\mathbb{S}^{n-1})}=(\bar{\delta}\alpha,\w)_{L^2(\mathbb{S}^{n-1})}=(\alpha,\bar{d}\w)_{L^2(\mathbb{S}^{n-1})}=0,$ which yields $\omega = 0$. To show the surjectivity of $\bar{d}$, we let $\omega\in V_{\boldsymbol{\lambda^{\Ex}_{k,p+1}}}\cap\im(\bar{d})$. It follows that $\bar{d}\bar{\delta}\omega + \underbrace{\bar{\delta} \bar{d}\omega}_{0} = \bar{\Delta}\omega = \boldsymbol{\lambda^{\Ex}_{k,p+1}}\omega$, which allows us to write $\omega = \bar{d}\left(\frac{1}{\boldsymbol{\lambda^{\Ex}_{k,p+1}}}\bar{\delta}\omega\right)$. Furthermore, we have $\bar{\delta}\left(\frac{1}{\boldsymbol{\lambda^{\Ex}_{k,p+1}}}\bar{\delta}\omega\right) = 0$ and $\bar{\Delta}\bar{\delta}\omega = \underbrace{\bar{d}\bar{\delta}\bar{\delta}\omega}_{0} + \bar{\delta} \bar{d}\bar{\delta}\omega = \bar{\delta}\boldsymbol{\lambda^{\Ex}_{k,p+1}}\omega = \boldsymbol{\lambda^{\Ex}_{k,p+1}}\bar{\delta}\omega$. Therefore the map $\bar{d}$ defined above is an isomorphism.\end{proof}
As a consequence, we have $\boldsymbol{\lambda^{\Coex}_{k,p}} = \boldsymbol{\lambda^{\Ex}_{k,p+1}}$. 
\begin{lemma}\cite[p.144]{IwasakiKiyoshi}\label{lm:spectredistinct}
    For $n\geq 3$, $p\in\{1,\cdots,n-2\}$, on $\Omega^p(\mathbb{S}^{n-1}),$ we have $\spec(\bar{\Delta}_{|\Ex})\cap\spec(\bar{\Delta}_{|\Coex})=\emptyset$, except for $n-1=2p.$
\end{lemma}
 
\begin{proof}
     First, recall that on $\Omega^p(\mathbb{S}^{n-1})$, $\spec({\bar{\Delta}}_{|\Ex})=\{(k+p)(k+n-p),\ k\in \mathbb{N}\}$. Furthermore, we have \begin{align*}
\spec({(\bar{\Delta}_p)}_{|\Coex}) 
&= \spec({(\bar{\Delta}_{p+1})}_{|\Ex}) \\
&= \left\{ (\tilde{k} + p + 1)(n - (p + 1) + \tilde{k}) \;\middle|\; \tilde{k} \in \mathbb{N} \right\} \\
&= \left\{ (\tilde{k} + p + 1)(\tilde{k} + n - p-1) \;\middle|\; \tilde{k} \in \mathbb{N} \right\} \\
&= \spec({(\bar{\Delta}_{n-p})}_{|\Ex}).
\end{align*} Note that, for the latter equality, this is due to the isomorphism induced by the Hodge star operator. Equating the two expressions for the spectra, we find\begin{align*}
    (k+p)(k+n-p)&=(\tilde{k} + p + 1)(\tilde{k} + n - p-1) \\
    \iff k^2+kn&=\tilde{k}^2+\tilde{k}n+n-1-2p.
\end{align*}
There are two cases to consider: \begin{itemize}
    \item \underline{If $n-1=2p$ : } We obtain $k=\tilde{k}.$
    \item \underline{If $n-1\neq 2p$ : } This case cannot occur, since the graphs of the functions\\ $y = x^2 + nx $ 
and $y = x^2 + nx + n-1 - 2p$ are disjoint. 
\end{itemize}
  \end{proof}
Therefore, on $\Omega^p(\mathbb{S}^{n-1})$, for $n-1\neq 2p$ we have \begin{align*}
    \spec(\bar{\Delta})&=\spec(\bar{\Delta}_{|\Ex})\sqcup \spec(\bar{\Delta}_{|\Coex})\\
    &= \{(k+p)(k+n-p)\mid k\in \mathbb{N}\} \sqcup \{(k+p+1)(k+n-p-1)\mid k\in \mathbb{N}\}.
       \end{align*}

In the following, we recall the multiplicities of the Laplacian eigenvalues on $\mathbb{S}^{n-1}$.

\begin{thm}\label{thm:vplaplacienmulti}\cite[Theorem, p.283]{GallotMeyer}, \cite[Theorem 6]{IwasakiKiyoshi}
For $n\geq3$, \( n-1 \neq 2p \) and $p\in\{1,\cdots,n-2\},$ we have on the one hand
 \begin{equation}\label{eq:dimEx}
\dim(\ker(\bar{\Delta}_p-\boldsymbol{\lambda^{\Ex}_{k,p}}))=\frac{(n + k-3)! (n + 2k-2)}{(p - 1)!  k!  (n - p-3)!  (n + k - p - 2 )(k + p)}
\end{equation}

and on the other hand  \begin{equation}\label{eq:dimCoex}
\begin{split}
\dim(\ker(\bar{\Delta}_p-\boldsymbol{\lambda^{\Coex}_{k,p}}))
&= \dim(\ker(\bar{\Delta}_{p+1}-\boldsymbol{\lambda^{\Ex}_{k,p+1}})) \\
&= \frac{(n + k-3)! (n + 2k-2)}{p!  k! (n - p - 4)!  (n + k - p-3)(k + p + 1)}.
\end{split}
\end{equation}

Furthermore, for \( n-1 = 2p \), we get
\begin{equation*}
\dim(\ker(\bar{\Delta}_p-\boldsymbol{\lambda^{\Ex}_{k,p}}))=\dim(\ker(\bar{\Delta}_p-\boldsymbol{\lambda^{\Coex}_{k,p}}))=\frac{2(2p + k)! (2p + 2k + 1)}{p! \, (p - 1)! \, k! \, (k + p)(k + p + 1)}.
\end{equation*}
  \end{thm}
  
\subsection{Spectrum of the NBS operator on the Euclidean ball}\label{sec:sectionNBSSpectre}
We start the computation of the spectrum of the NBS on the Euclidean ball. To do so, we let $(M^n,g)=(\mathbb{B}^n,g_0)$ be the Euclidean ball of dimension $n\geq 2$ and radius $1$ equipped with the standard metric. We study the spectrum of the NBS problem: \begin{equation}\label{eq:NBSspectre}
    (NBS)
    \left\{
    \hspace{-0.3em}
    \begin{array}{r@{\quad}c@{\quad}l@{\quad}l}
        \Delta^2 \w &=& 0 & \text{on } M \\[0.15em]
        \nu \lrcorner \w &=& 0 & \text{on } \partial M \\[0.15em]
        \nu \lrcorner d\w &=& 0 & \text{on } \partial M \\[0.15em]
        \nu \lrcorner \Delta\w &=& 0 & \text{on } \partial M \\[0.15em]
        \nu \lrcorner d\Delta\w + \ell \, \iota^* \w &=& 0 & \text{on } \partial M
    \end{array}
    \right.
\end{equation} \begin{remark}\label{rk:spectreNBSFct}
 The spectrum of NBS for functions $f$ on the interval $M=[0,1]$ reduces to $\{0, 24\}$, with each of these eigenvalues having multiplicity $1$. We get this by simply solving the system \begin{equation*}
    \left\{
    \hspace{-0.3em}
    \begin{array}{r@{\quad}c@{\quad}l@{\quad}l}
        f^{(4)} &=& 0 & \text{on } M \\[0.15em]
        f'(0) &=& 0 & \text{on } \partial M \\[0.15em]
        f'(1) &=& 0 & \text{on } \partial M \\[0.15em]
        -f^{(3)}(0) + \ell f(0) &=& 0 & \text{on } \partial M \\[0.15em]
        f^{(3)}(1) + \ell f(1) &=& 0 & \text{on } \partial M.
    \end{array}
    \right.
\end{equation*}
\end{remark}
The purpose of this section is to show \thref{thm:spectreNBS} and \thref{thm:spectreNBScasparticulier}.
\begin{proof}[Proof of \thref{thm:spectreNBS}]
Assume $n\geq3$ and $p\in\{1,\cdots,n-2\}.$ We start the proof by recalling Almansi decomposition \cite{Almansi}, \cite[Chapter V]{aronszajn}, which states that for any biharmonic form $\omega$ on the Euclidean ball, there exist harmonic forms $\w_0, {\w_1} \in \Omega^p(\mathbb{B}^n)$ with coordinate components $\w_0 = \sum_I \w_{0,I} dx_I$ and ${\w_1} = \sum_I \w_{1,I} dx_I$, where $I = (1, \dots, p)$, such that:
\begin{equation*}
\omega(x) = \w_0(x) + |x|^2 {\w_1}(x).
\end{equation*}
We define ${\w_2}\in \Omega^p(\mathbb{B}^n)$ by ${\w_2}:= -\frac{1}{2}(\w_0+{\w_1})$. Since ${\w_2}$ is harmonic, substituting this into the classical expansion yields the modified decomposition:
\begin{equation}\label{eq:w=W+H}
\omega(x) = -2{\w_2}(x) + (|x|^2 - 1) {\w_1}(x).
\end{equation}

  We consider $p$-forms ${\w}$ on $\mathbb{B}^{n}$ which can be written in polar coordinates for all $(r,\theta)\in]0,1[\times \mathbb{S}^{n-1}$ as \begin{equation}\label{eq:W=P+Q}
{\w}(r,\theta) = P(r) \, dr \wedge \eta(\theta) + Q(r) \xi(\theta)\end{equation} where $\eta\in \Omega^{p-1}(\mathbb{S}^{n-1})$ and $\xi \in \Omega^{p}(\mathbb{S}^{n-1})$ are independent of $r$, and $P, Q\in$ are smooth functions in $C^\infty([0,1],\mathbb{R}_+)$ depending only on $r$. 

\begin{lemma}\label{lm:P0Q0}
   Let $\w\in\Omega^p(\mathbb{B}^n)$ be a smooth form which can be written as in \eqref{eq:W=P+Q}. Then for $p\geq1,$ $Q(0)=0$ and for $p\geq2,$ $P(0)=0.$
\end{lemma}
\begin{proof}
Let $
\phi(x) = \left( |x|, \frac{x}{|x|} \right) \in \mathbb{R}^*_+ \times \mathbb{S}^{n-1}$ then $d_x\phi(X) = \left( \left\langle X, \frac{x}{|x|} \right\rangle, \frac{1}{|x|} \left( X - \left\langle X, \frac{x}{|x|} \right\rangle \frac{x}{|x|} \right) \right).
$
We consider $\w\in \Omega^p(\mathbb{B}^n)$ satisfying \eqref{eq:W=P+Q}, thus denoting by $X^T=X - \left\langle X, \frac{x}{|x|} \right\rangle \frac{x}{|x|}$, we have:
\begin{align*}
(\phi^*\omega)_x(X_1, \cdots, X_p) &= \omega_{\phi(x)}(d_x\phi(X_1), \cdots, d_x\phi(X_p)) \\
&= \omega_{\phi(x)} \left( \left\langle X_1, \frac{x}{|x|} \right\rangle \partial_r + \frac{1}{|x|} X_1^T, \cdots, \left\langle X_p, \frac{x}{|x|} \right\rangle \partial_r + \frac{1}{|x|} X_p^T \right) \\
&= \sum_{j=1}^p \left\langle X_j, \frac{x}{|x|} \right\rangle \omega_{\phi(x)} \left( \frac{1}{|x|} X_1^T, \cdots, \underbrace{\partial_r}_{j}, \cdots, \frac{1}{|x|} X_p^T \right) \\
&\quad + \frac{1}{|x|^p} \omega_{\phi(x)} \left( X_1^T, \cdots, X_p^T \right) \\
&= |x|^{-p+1} \sum_{j=1}^p (-1)^{j-1} \left\langle X_j, \frac{x}{|x|} \right\rangle \partial_r \lrcorner \omega_{\phi(x)} (X_1^T, \cdots, \widehat{X_j^T}, \cdots, X_p^T) \\
&\quad + |x|^{-p} \omega_{\phi(x)} (X_1^T, \cdots, X_p^T) \\
&= |x|^{-p+1} P(|x|) \sum_{j=1}^p (-1)^{j-1} \left\langle X_j, \frac{x}{|x|} \right\rangle \eta_{\frac{x}{|x|}} (X_1^T, \cdots, \widehat{X_j^T}, \cdots, X_p^T) \\
&\quad + |x|^{-p} Q(|x|) \xi_{\frac{x}{|x|}} (X_1^T, \cdots, X_p^T).
\end{align*}
In particular, fixing some $y\in\mathbb{S}^{n-1},$ and taking $x= ty, t \in \mathbb{R}$, we have two cases:
\begin{itemize}
    \item [$\bullet$] If $X_1 = y$ and $X_2,\cdots,X_p\perp y$: $$(\phi^*\omega)_{ty}(y, X_2, \cdots, X_p) = |t|^{-p+1} P(|t|){\operatorname{sgn(t)}}\eta_{\operatorname{sgn(t)}y})(X_2, \cdots, X_p).$$ Since there exist $y\in\mathbb{S}^{n-1}$ and $X_2,\cdots,X_k$ such that $\eta_{ y}(X_2, \cdots, X_p) \neq 0$ or $\eta_{-y}(X_2, \cdots, X_p) \neq 0$, the limit when $t$ goes to zero of $|t|^{-p+1} P(|t|) $ must be finite, consequently $P(t) = t^{p-1} \tilde{P}(t)$ for some $\tilde{P}(t)$. In particular for $p=1,$ this case is empty.
    \item [$\bullet$] If $ X_1, \cdots, X_p \perp y$: $(\phi^*\omega)_{ty}(y, X_2, \cdots, X_p) = |t|^{-p} Q(|t|) \xi_{\operatorname{sgn(t)}y}(X_1, \cdots, X_p)$. Analogously, we get $ Q(t) = t^p \tilde{Q}(t)$ for some $\tilde{Q}(t).$
\end{itemize}
\end{proof}
We now recall the following lemma, which provides an expression for the Hodge Laplacian on the unit ball in $\mathbb{R}^n$ as in \cite{RaulotSavo3}:
 \begin{lemma}\cite[Lemma 6]{RaulotSavo3}\label{lm:delta=deltat+deltar} Let $\bar{d}$ and $\bar{\delta}$ denote, respectively, the differential and the co-differential acting on $\mathbb{S}^{n-1}$. Let ${\w}$ be a $p$-form that can be written $\w = P\, dr \wedge \eta + Q \xi$ as in \eqref{eq:W=P+Q}. Then we have:
 $$\Delta {\w}={(\Delta\w)}_{|\tann}+dr\wedge {(\Delta\w)}_{|\radd}$$ where\[
\left\{
\begin{alignedat}{2}
    {(\Delta\w)}_{|\tann} &= \iota^*\Delta\w &&= \frac{Q_{\w}}{r^2}\bar{\Delta}\xi_{\w}-\left(Q''_{\w}+\frac{n-2p-1}{r}Q'_{\w}\right)\xi_{\w}-\frac{2P_{\w}}{r}\bar{d}\eta_{\w} \\[4mm]
    {(\Delta\w)}_{|\radd} &= \partial r\lrcorner\Delta\w &&= \frac{P_{\w}}{r^2}\bar{\Delta}\eta_{\w}-\left(P'_{\w}+\frac{n-2p+1}{r}P_{\w}\right)'\eta_{\w}-\frac{2Q_{\w}}{r^3}\bar{\delta}\xi_{\w}.
\end{alignedat}
\right.
\]
\end{lemma}
Let us recall the following integration by parts formula: \begin{align}
\int_M \langle \Delta^2 \omega , \omega' \rangle \, d\mu_g
&= \int_M \langle \Delta \omega , \Delta \omega' \rangle \, d\mu_g \notag \\
&\quad + \int_{\partial M} \Big[
    \langle \nu \lrcorner d \Delta \omega, \iota^* \omega' \rangle
    - \langle \iota^* \Delta \omega, \nu \lrcorner d \omega' \rangle \notag \\
&\qquad\qquad
    + \langle \nu \lrcorner \Delta \omega, \iota^* \delta \omega' \rangle
    - \langle \iota^* \delta \Delta \omega, \nu \lrcorner \omega' \rangle
\Big] \, d\mu_g.
\label{eq:IPP2}
\end{align} For more details the reader can refer to \cite{Article1}.
\begin{lemma}\label{lm:wnonnul}
   Let $\w\in\Omega(\mathbb{B}^n)$ be an eigenform of NBS associated to a non-zero eigenvalue. Assume that $\w=-2\w_2(x)+(\lvert x\rvert^2-1)\w_1(x)$ where for $i\in\{1,2\},$ the form $\w_i$ is harmonic. Then ${\w_1}$ and ${\w_2}$ are both non-zero.
\end{lemma}

\begin{proof}
    If ${\w_2}$ is zero, then $\omega_{|\partial M}=0$. By integration by parts \eqref{eq:IPP2}, we obtain $\Delta\omega=0$, thus, by \cite{ColetteAnne}, $\omega=0$ on $M$, which is a contradiction. Similarly, if ${\w_1}=0$, then $\omega$ is harmonic and satisfies the boundary conditions of NBS \eqref{eq:NBSspectre}, therefore, ${\w_2}\in H_{\Abs}(\mathbb{B}^n)=\{0\}$ which implies $\omega=0$ on $M.$ Therefore, we necessarily have ${\w_2}\neq0$ and ${\w_1}\neq0.$
\end{proof} 
Let us now fix an eigenform $\w\in\Omega(\mathbb{B}^n)$ for NBS associated to a non-zero eigenvalue and satisfying the hypothesis of the previous Lemma. Assume moreover that $\w_1$ and $\w_2$ can be written as on \eqref{eq:W=P+Q}. Let us fix $i\in\{1,2\}.$ By \lemref{lm:delta=deltat+deltar}, we have: \begin{align}
    \Delta \w_i = 0 
    &\iff (\Delta\w_i)_{|\tann} = 0 \quad\text{and}\quad (\Delta\w_i)_{|\radd} = 0 \nonumber \\[1ex]
    &\iff \left\{
    \begin{aligned}
        (E_{\tann})  &: \frac{Q_{\w_i}}{r^2}\bar{\Delta}\xi_{\w_i} = Q''_{\w_i}\xi_{\w_i} + \frac{n-2p-1}{r}Q'_{\w_i}\xi_{\w_i} + \frac{2P_{\w_i}}{r}\bar{d}\eta_{\w_i} \\[1ex]
        (E_{\radd}) &: \begin{aligned}[t]
            \frac{P_{\w_i}}{r^2}\bar{\Delta}\eta_{\w_i} = P''_{\w_i}\eta_{\w_i} &+ \frac{n-2p+1}{r}P'_{\w_i}\eta_{\w_i} \\
            &- \frac{n-2p+1}{r^2}P_{\w_i}\eta_{\w_i} + \frac{2Q_{\w_i}}{r^3}\bar{\delta}\xi_{\w_i}.
        \end{aligned}
    \end{aligned}
    \right. \label{eq:AWBW}
\end{align}
Applying $\bar{d}$ and $\bar{\delta}$ to $(E_{\tann})$ and $(E_{\radd})$, respectively, we obtain the following system:
\begin{equation}\label{eq:DAWDQW}
    \left\{
    \hspace{-0.3em}
    \begin{array}{l@{\quad}r@{\quad}c@{\quad}l}
        \bar{d}(E_{\tann}): & \dfrac{Q_{\w_i}}{r^2} \bar{d}\bar{\Delta}\xi_{\w_i} &=& Q''_{\w_i}\bar{d}\xi_{\w_i} + \dfrac{n-2p-1}{r} Q'_{\w_i}\bar{d}\xi_{\w_i} \\[2mm]
        \bar{\delta}(E_{\radd}): & \dfrac{P_{\w_i}}{r^2} \bar{\delta}\bar{\Delta}\eta_{\w_i} &=& P''_{\w_i}\bar{\delta}\eta_{\w_i} + \dfrac{n-2p+1}{r} P'_{\w_i}\bar{\delta}\eta_{\w_i} - \dfrac{n-2p+1}{r^2}\bar{\delta}\eta_{\w_i}P_{\w_i}
    \end{array}
    \right.
\end{equation}
Isolating $\bar{d}\xi_{\w_i}$ and $\bar{\delta}\eta_{\w_i}$ in $\bar{d}(E_{\tann})$ and $\bar{\delta}(E_{\radd})$, respectively, we obtain:
\begin{equation}\label{eq:DAWDQWFinale}
    \left\{
    \hspace{-0.3em}
    \begin{array}{l@{\quad}r@{\quad}c@{\quad}l}
        \bar{d}(E_{\tann}): & \bar{\Delta}\bar{d}\xi_{\w_i} &=& \dfrac{r^2Q''_{\w_i} + r(n-2p-1)Q'_{\w_i}}{Q_{\w_i}} \bar{d}\xi_{\w_i} \\[3mm]
        \bar{\delta}(E_{\radd}): & \bar{\Delta}\bar{\delta}\eta_{\w_i} &=& \dfrac{r^2P''_{\w_i} + r(n-2p+1)P'_{\w_i} - (n-2p+1)P_{\w_i}}{P_{\w_i}} \bar{\delta}\eta_{\w_i}
    \end{array}
    \right.
\end{equation}
Consequently, $\bar{d}\xi_{\w_i}$ and $\bar{\delta}\eta_{\w_i}$ are both eigenforms of the Laplacian on the sphere, unless they vanish. By \thref{thm:vplaplacien} there exist $k,k'\in\mathbb{N}$ such that writing\[
\begin{aligned}
    \boldsymbol{\lambda_{k,p+1}^{\Ex}} &= (k+p+1)(k+n-p-1) 
    &;\,\,  \,
    \boldsymbol{\lambda_{k',p-2}^{Coex}} &= (k'+p-1)(k'+n-p+1)
\end{aligned}
\]we have $\bar{\Delta}(\bar{d}\xi_{\w_i})=\boldsymbol{\lambda_{k,p+1}^{Ex}}\bar{d}\xi_{\w_i}$ if $\bar{d}\xi_{\w_i}\neq0$ and $\bar{\Delta}(\bar{\delta}\eta_{\w_i})= \boldsymbol{\lambda_{k',p-2}^{Coex}}\bar{\delta}\eta_{\w_i}$ if $\bar{\delta}\eta_{\w_i}\neq0.$ We therefore obtain the following system:\begin{equation}\label{eq:systemeQuiDonnePetQpolynomes}
    \left\{
    \hspace{-0.3em}
    \begin{array}{r@{\quad}c@{\quad}l}
        Q''_{\w_i} + \frac{n-2p-1}{r}Q'_{\w_i} - \frac{\boldsymbol{\lambda_{k,p+1}^{\Ex}}}{r^2}Q_{\w_i} &=& 0 \\[3mm]
        P''_{\w_i} + \frac{n-2p+1}{r}P'_{\w_i} - \frac{n-2p+1 + \boldsymbol{\lambda_{k',p-2}^{\Coex}}}{r^2}P_{\w_i} &=& 0
    \end{array}
    \right.
\end{equation}

\begin{lemma}\label{lem:PetQpolynomes}
The solution of \eqref{eq:systemeQuiDonnePetQpolynomes} such that $Q_{\w_i}(r)$ vanishes and $P_{\w_i}(r)$ is well defined at $r=0$ are $Q_{\w_i}(r) = c_1 r^{k+p+1}$ and $P_{\w_i}(r) = c_2 r^{k'+p}$. Moreover, \eqref{eq:AWBW} implies $k=k'.$
\end{lemma}
\begin{proof}
 Recall that a differential equation of the form\begin{equation}\label{eq:F} r^2F''(r)+arF'(r)+bF(r)=0, \ \text{on}\ [0,1]
    \end{equation} has $F(r)=r^\alpha$ (where $\alpha\geq1$) as a solution if and only if \begin{equation*}
    \left\{
    \hspace{-0.3em}
    \begin{array}{r@{\quad}c@{\quad}l@{\quad}l}
        b &=& 0 & \text{if } \alpha = 0 \\[0.15em]
        a + b &=& 0 & \text{if } \alpha = 1 \\[0.15em]
        \alpha^2 + (a - 1)\alpha + b &=& 0 & \text{if } \alpha \notin \{0, 1\}.
    \end{array}
    \right.
\end{equation*}In the last case, $\Delta=(a-1)^2-4b$ and $\alpha=\frac{-(a-1)\pm \sqrt{(a-1)^2-4b}}{2}$ provided that $\Delta>0$.

\begin{itemize}
    \item \underline{\bf{Computation of $Q$}:} In this case, $a_Q=n-2p-1$ and $b_Q=-\boldsymbol{\lambda_{k,p+1}^{Ex}}=-(k+p+1)(k+n-p-1).$ Hence, the discriminant is given by: \begin{align*}
        \Delta_Q&=(n-2p-2)^2+4(k+p+1)(k+n-p-1)\\
        &=n^2+4p^2+4-4n-4np+8p\\
        &+4(k+kp+k^2+n+np+nk-1-p-k-p-p^2-kp)\\
        &=n^2+4k^2+4nk=(n+2k)^2>0.
    \end{align*}
    Therefore we get the following roots: \[\alpha_{1,Q}=\frac{-n+2p+2+n+2k}{2}=k+p+1 \] \[\text{and } \alpha_{2,Q}=\frac{-n+2p+2-n-2k}{2}=-n-k+p+1.\] Since $k\geq0$ and $p\in\{1,\cdots,n-1\},$ $\alpha_{2,Q}<0$ is thus rejected because $Q(0)=0$ by \lemref{lm:P0Q0}. Therefore, $Q_{\w_2}(r)=r^{k+p+1}.$
    \item \underline{\bf{Computation of $P$}:} In this case, $a_P=n-2p+1$ and $b_P=-(n-2p+1 +\boldsymbol{\lambda_{k',p-2}^{Coex}})=-n+2p-1-(k'+p-1)(k'+n-p+1)$. Hence, the discriminant is given by: \begin{align*}
   \Delta_P&=(n-2p)^2+4n-8p+4+4(k'+p-1)(k'+n-p+1)\\
   &=n^2-4np+4p^2+4n-8p+4\\
   &\quad+4(k'n-k'p+k'{^2}+k'+np-p^2+k'p+p-n+p-k'-1)\\
   &=n^2+4k'n+k^2=(n+2k')^2>0. \end{align*}  Therefore we get the following roots: \[\alpha_{1,P}=\frac{-n+2p+n+2k'}{2}=p+k'\ ; \ \alpha_{2,P}=\frac{-n+2p-n-2k'}{2}=-n-k'+p.\] The root $\alpha_{2,P}<0$ and is therefore rejected. Thus $P_{\w_2}(r)=r^{k'+p}.$
\end{itemize} Plugging these values into $(E_{\tan})$ of \eqref{eq:AWBW} yields $k = k'$.
\end{proof}
Using \eqref{eq:w=W+H}, we formulate the boundary conditions of the NBS \eqref{eq:NBSspectre} in terms of $P_{\w_i}, Q_{\w_i}, \xi_{\w_i}$, and $\eta_{\w_i}$ for $i\in\{1,2\}$.
\begin{prop}\label{prop:conditions}
Following the same notations as above, we obtain the following expressions for the boundary conditions of the NBS \eqref{eq:NBSspectre}:\[
\left\{
\begin{aligned}
    \nu\lrcorner \w &= 0 \\[0.6em]
    \nu\lrcorner d\w &= 0 \\[0.6em]
    \nu\lrcorner\Delta \w &= 0 \\[0.6em]
    \nu \lrcorner d\Delta \w+\ell\iota^* \w &= 0
\end{aligned}
\right.
\iff
\left\{
\begin{aligned}
    P_{\w_2}(1)&=0 \text{ or } \eta_{\w_2}=0 \text{ on } \mathbb{S}^{n-1} \\[0.6em]
    Q'_{\w_2}(1)\xi_{\w_2}&=Q_{\w_1}(1)\xi_{\w_1} \\[0.6em]
    \eta_{\w_1}&=0 \text{ or } (2n-4p+4)P_{\w_1}(1)+4P'_{\w_1}(1)=0 \\[0.6em]
    \xi_{\w_1} \bigl[(4&-4p+2n)Q'_{\w_1}(1)+4Q''_{\w_1}(1) \bigr] \\
    &-2\ell Q_{\w_2}(1)\xi_{\w_2}=0.
\end{aligned}
\right.
\]
\end{prop}
\begin{remark}\label{rk:condition+}
    It follows from the proof this proposition is valid for $p=n-1.$
\end{remark}
Before proving \propref{prop:conditions}, we first establish the following two lemmas, which will be used in the computation:
\begin{lemma}
  Let $\xi\in\Omega^p(\mathbb{B}^n\setminus\{0\})$ be a $p$-form on $\mathbb{B}^n\setminus\{0\}$, that in polar coordinates can be identified to a $p$-form on the sphere $\mathbb{S}^{n-1}$. Then on $\mathbb{B}^n\setminus\{0\}$, we have \begin{equation}\label{eq:nabla=pr}
       \nabla^{\mathbb{B}^n}_{\partial_r}\xi=-\frac{p}{r}\xi.
   \end{equation}  
\end{lemma}
\begin{proof}
    Let $\xi=v_1^\flat\wedge\ldots\wedge v_p^\flat,$ with $v_i\in T\mathbb{S}^{n-1}$. For $Y\in T\mathbb{S}^{n-1}$ we have \begin{align*}
(\nabla_{\partial_r}^{\mathbb{B}^n}v_i^\flat)(Y)
&= \underbrace{\partial_r\big(v_i^\flat(Y)\big)}_{0}
 - v_i^\flat\Big(\underbrace{\nabla_{\partial_r}^{\mathbb{B}^n} Y}_{\frac{1}{r}Y}\Big) \\
&= -\frac{1}{r} v_i^\flat(Y).
\end{align*} Therefore, \begin{align*}
(\nabla^{\mathbb{B}^n}_{\partial_r}\xi)_{(r,x)}
&= \sum_{i=1}^p 
v_1^\flat \wedge \cdots \wedge 
\nabla_{\partial_r}^{\mathbb{B}^n} v_i^\flat \wedge \cdots \wedge v_p^\flat \\
&= -\sum_{i=1}^p \frac{1}{r}
\bigl( v_1^\flat \wedge \cdots \wedge v_p^\flat \bigr) \\
&= -\frac{p}{r}\,\xi.
\end{align*}
\end{proof}
Recall that we denote by $\nabla$ the Euclidean connection on $\mathbb{R}^n$ and by $\bar{\nabla}$ the Riemannian connection on $\mathbb{S}^{n-1}$.
\begin{lemma}
Let $\xi\in\Omega^p(\mathbb{B}^n\setminus\{0\})$ be a $p$-form that in polar coordinates can be identified to a $p$-form on the sphere $\mathbb{S}^{n-1}.$ Then for any vector field $X$ on $\mathbb{B}^n\setminus\{0\},$ which is tangent to $\mathbb{S}^{n-1}=\partial\mathbb{B}^n$, we have on $\mathbb{S}^{n-1}:$
\begin{equation}\label{eq:nablax}
\nabla_X\xi=\nu^\flat\wedge(SX\lrcorner\xi)+\bar{\nabla}_X\iota^*\xi.\end{equation} 
\begin{proof} Let $X_1,\cdots,X_p$ be vector fields on $\mathbb{S}^{n-1}$. We compute
\begin{align*}
    (\nu\lrcorner\nabla_X\xi)(X_2,\cdots,X_p)&=(\nabla_X\xi)(\nu,X_2,\cdots,X_p)\\
    &=X\underbrace{\left(\xi(\nu,X_2,\cdots,X_p)\right)}_{0}-\xi(\nabla_X\nu,X_2,\cdots,X_p)\\
    &\quad-\sum_{j=2}^p\underbrace{\xi(\nu,\cdots,\nabla_XX_j,\cdots,X_p)}_{0}\\
    &=(SX\lrcorner\xi)(X_2,\cdots,X_p).
\end{align*}Since $\partial_r\lrcorner\xi=0$, we have:
\begin{align*}
(\nabla_X \xi)(X_1, \cdots, X_p) &= X(\xi(X_1, \cdots, X_p)) - \sum_{j=1}^p \xi(X_1, \cdots, \nabla_X X_j, \cdots, X_p) \\
&= X(\xi(X_1, \cdots, X_p)) \\
&\quad- \sum_{j=1}^p \xi(X_1, \cdots, \bar{\nabla}_X X_j + \langle SX, X_j \rangle \nu, \cdots, X_p\\&= X(\xi(X_1, \cdots, X_p)) - \sum_{j=1}^p \xi(X_1, \cdots, \bar{\nabla}_X X_j, \cdots, X_p)\\
&\quad - \sum_{j=1}^p \langle SX, X_j \rangle \xi(X_1, \cdots, \nu, \cdots, X_p)\\
&= (\bar{\nabla}_X \xi)(X_1, \cdots, X_p).
\end{align*}
Therefore, \begin{align*}
    \nabla_X\xi&=\nu^\flat\wedge(\nu\lrcorner\nabla_X\xi)+\iota^*(\nabla_X\xi)\\
    &=\nu^\flat\wedge(SX\lrcorner\xi)+\bar{\nabla}_X\xi.
\end{align*}
\end{proof}
\end{lemma}
\begin{proof}[Proof of \propref{prop:conditions}]
We now begin the computation of the boundary conditions.\begin{itemize}[label=\textbullet]
    \item $\underline{\bm{\nu \lrcorner \w=0 :}}$ 
    \begin{align*}
        \nu \lrcorner \w = 0 
        &\iff -2 \nu \lrcorner {\w_2}+ (\underbrace{\lvert x \rvert^2}_{1} - 1) \nu \lrcorner {\w_1} = 0 \\
        &\iff \nu \lrcorner {\w_2}= 0.
    \end{align*}
Then, using \eqref{eq:W=P+Q}, we get
\begin{align}
    \nu\lrcorner {\w_2}=0 
    &\iff Q_{\w_2}\underbrace{\nu\lrcorner\xi_{\w_2}}_{0}+ P_{\w_2}\nu\lrcorner(dr\wedge\eta_{\w_2})=0  \nonumber\\  \nonumber
    &\iff P_{\w_2}\underbrace{dr(\nu)}_{-1}\eta_{\w_2}=0\\ \nonumber
    &\iff P_{\w_2}\eta_{\w_2}=0 \ \text{on}\ \partial M\\ 
    &\iff P_{\w_2}(1)=0 \ \text{or}\ \eta_{\w_2}=0\ \text{on}\ \mathbb{S}^{n-1}. \label{eq:condition1}
\end{align}
\item $\underline{\bm{\nu \lrcorner d\w =0:}}$ 
    \begin{align*}
        \nu\lrcorner d\w=0 
        &\iff -2\nu\lrcorner d {\w_2}+ \nu\lrcorner d (\lvert x\rvert ^2-1){\w_1}=0\\
        &\iff -2\nu\lrcorner d {\w_2}+ 2\nu\lrcorner (x^\flat\wedge {\w_1}) + (\underbrace{\lvert x \rvert^2}_{1} - 1) \nu \lrcorner d {\w_1} = 0\\
        &\iff -2\nu\lrcorner d {\w_2}+ 2\underbrace{x^\flat(\nu)}_{-1}({\w_1})-2x^\flat\wedge (\nu\lrcorner {\w_1})=0\\
        &\iff -2\nu\lrcorner d {\w_2}-2{\w_1}-2\underbrace{x^\flat}_{-\nu^\flat}\wedge(\nu\lrcorner {\w_1})=0\\
        &\iff \nu\lrcorner d {\w_2}+ \iota^* {\w_1}=0.
    \end{align*}
Using \eqref{eq:W=P+Q} again, we get: \begin{align*}
    d {\w_2}&=d Q_{\w_2}\wedge\xi_{\w_2}+Q_{\w_2}d\xi_{\w_2}+P'_{\w_2}(r) \underbrace{dr\wedge dr}_{0}\wedge \eta_{\w_2}\\
    &+P_{\w_2}\underbrace{d(dr)}_{0}\wedge\eta_{\w_2}-P_{\w_2}dr\wedge d\eta_{\w_2}\\
    &=d Q_{\w_2}\wedge\xi_{\w_2}+Q_{\w_2}d\xi_{\w_2}-P_{\w_2}dr\wedge d\eta_{\w_2}.
\end{align*} 
We now compute $d\xi_{\w_2}$: \begin{align*}
    d\xi_{\w_2}&=\sum_{j=1}^n e_j^\flat\wedge\nabla_{e_j}\xi_{\w_2}\\
    &=\sum_{j=1}^{n-1} e_j^\flat\wedge\nabla_{e_j}\xi_{\w_2}+ dr\wedge{\nabla_{\partial_r}\xi_{\w_2}}\\
    &\hspace{-0.8em}\overset{\eqref{eq:nabla=pr}}{=}\sum_{j=1}^{n-1} e_j^\flat\wedge\nabla_{e_j}\xi_{\w_2}-\frac{p}{r} dr\wedge\xi_{\w_2}\\
    & \hspace{-0.8em}\overset{\eqref{eq:nablax}}{=}\sum_{j=1}^{n-1}\left(e_j^\flat\wedge(\nu^\flat \wedge(Se_j\lrcorner\xi_{\w_2}))+e_j^\flat\wedge\bar{\nabla}_{e_j}\iota^*\xi_{\w_2}-e_j^\flat\wedge(Se_j)^\flat\wedge(\underbrace{\nu\lrcorner\xi_{\w_2}}_{0})\right)\\
    &\qquad-\frac{p}{r} dr\wedge\xi_{\w_2}\\
    &=-\nu^\flat\wedge\underbrace{(\sum_{j=1}^{n-1}e_j^\flat\wedge(Se_j\lrcorner\xi_{\w_2}))}_{S^{[p]}\xi_{\w_2}}+\bar{d}\iota^*\xi_{\w_2}-\frac{p}{r} dr\wedge\xi_{\w_2}\\
    &=-\frac{p}{r}\nu^\flat\wedge\xi_{\w_2}+\bar{d}\iota^*\xi_{\w_2}-\frac{p}{r} dr\wedge\xi_{\w_2}\\
    &=-\frac{p}{r}(\nu^\flat+\underbrace{dr}_{-\nu^\flat})\wedge\xi_{\w_2}+\bar{d}\iota^*\xi_{\w_2}\\
    &=\bar{d}\iota^*\xi_{\w_2}. 
\end{align*}
Similarly, we obtain ${d}\eta_{\w_2}=\bar{d}\iota^*\eta_{\w_2}.$  Hence, we have\begin{align*}
    d{\w_2}=dr\wedge (Q'_{\w_2}(r)\xi_{\w_2}-P_{\w_2}\bar{d}\iota^*\eta_{\w_2})+Q_{\w_2}\bar{d}\iota^*\xi_{\w_2}.
\end{align*}
Consequently, \begin{equation*}
    \nu\lrcorner d{\w_2}=-\iota^*{\w_1} \iff Q'_{\w_2}(1)\xi_{\w_2}-P_{\w_2}(1)\bar{d}\iota^*\eta_{\w_2}=Q_{\w_1}(1)\xi_{\w_1}.
\end{equation*}
Finally, using \eqref{eq:condition1}, we get the following condition\begin{equation}
    \nu\lrcorner d{\w_2}=-\iota^*{\w_1} \iff Q'_{\w_2}(1)\xi_{\w_2}=Q_{\w_1}(1)\xi_{\w_1}.\label{eq:condition2}
\end{equation}
\item $\underline{\bm{\nu \lrcorner \Delta\w =0:}}$ 
 We start by computing $\Delta\w.$ \begin{align}\label{eq:danslapreuvedesconditions}
        \Delta\w&=-2\underbrace{\Delta {\w_2}}_{0}+\Delta\left((\lvert x \rvert^2 - 1){\w_1}\right)\nonumber\\ \nonumber
        &=\Delta(\lvert x\rvert^2-1)\cdot {\w_1}-2\nabla_{\nabla(\lvert x\rvert^2-1)}{\w_1}+(\lvert x\rvert^2-1)\underbrace{\Delta {\w_1}}_{0}\nonumber\\
        &=-2n\cdot {\w_1}-4\nabla_x {\w_1}.
    \end{align}
    Hence we get, \begin{align*}
        \nu\lrcorner\Delta\w=0 &\iff -2n\nu\lrcorner {\w_1}+4\nu\lrcorner\nabla_\nu {\w_1}=0. \\
   \end{align*}
 We evaluate each term separately. We begin with $\nabla_\nu {\w_1}$. \begin{align}\label{eq:calculConditionsdebord}
       \nabla_\nu {\w_1}=-\nabla_{\partial_r} {\w_1}&=-\nabla_{\partial_r}(P_{\w_1}dr\wedge \eta_{\w_1}+Q_{\w_1}\xi_{\w_1})_{|r=1}\nonumber\\\nonumber
       &=-P'_{\w_1}(1)dr\wedge\eta_{\w_1}-Q'_{\w_1}(1)\xi_{\w_1}-Q_{\w_1}(1)\nabla_{\partial_r}\xi_{\w_1}\\ \nonumber
       &\quad-P_{\w_1}(1)dr\wedge\nabla_{\partial_r}\eta_{\w_1}\nonumber\\ \nonumber 
       &\overset{\eqref{eq:nabla=pr}}{=}-P'_{\w_1}(1)dr\wedge\eta_{\w_1}-Q'_{\w_1}(1)\xi_{\w_1}+pQ_{\w_1}(1)\xi_{\w_1}\\ \nonumber
       &\quad+P_{\w_1}(1)dr\wedge(p-1)\eta_{\w_1} \nonumber\\\nonumber
       &=\xi_{\w_1}(-Q'_{\w_1}(1)+pQ_{\w_1}(1))\\
       &\quad+dr\wedge\eta_{\w_1}(-P'_{\w_1}(1)+P_{\w_1}(1)(p-1)).
   \end{align} 
   This implies that $\nu\lrcorner\nabla_\nu {\w_1}=\eta_{\w_1}(P'_{\w_1}(1)-P_{\w_1}(1)(p-1)).$
   We also determine \[\nu\lrcorner {\w_1}=-\partial_r\lrcorner(Q_{\w_1}(r)\xi_{\w_1}+P_{\w_1}(r) dr\wedge\eta_{\w_1})=-P_{\w_1}(1)\eta_{\w_1}.\]
  Therefore, we conclude the following \begin{align} 
       \nu\lrcorner\Delta\w=0 &\iff 2nP_{\w_1}(1)\eta_{\w_1}+4\eta_{\w_1}(P'_{\w_1}(1)-P_{\w_1}(1)(p-1))=0\nonumber \\ \nonumber
       &\iff \eta_{\w_1}\big((2n-4p+4)P_{\w_1}(1)+4P'_{\w_1}(1)\big)=0\nonumber\\ 
       &\iff \eta_{\w_1}=0\ \text{or}\ (2n-4p+4)P_{\w_1}(1)+4P'_{\w_1}(1)=0.\label{eq:condition3}
   \end{align}
 \item $\underline{\bm{\nu \lrcorner d\Delta\w+\ell\iota^*\w =0:}}$
     By \eqref{eq:danslapreuvedesconditions}, we have \begin{align*}
          \nu \lrcorner d\Delta\w+\ell\iota^*\w=0&\iff -2n\nu\lrcorner d {\w_1}-4\nu\lrcorner d(\nabla_x {\w_1})\\
          &\qquad\quad-2\ell\iota^*{\w_2}+\ell\iota^*(\underbrace{\lvert x\rvert^2}_{1}-1){\w_1}=0\\
          &\iff 2n\partial_r\lrcorner d {\w_1}+4\partial_r\lrcorner d(\nabla_x {\w_1})-2\ell\iota^*{\w_2}=0.
      \end{align*}
      We compute each term individually: \begin{align*}
    d {\w_1}&=Q'_{\w_1} dr\wedge\xi_{\w_1} +Q_{\w_1} \bar{d}\iota^*\xi_{\w_1}+P'_{\w_1}(r) \underbrace{dr\wedge dr}_{0}\wedge \eta_{\w_1}\\
    &\quad+P_{\w_1}\underbrace{d(dr)}_{0}\wedge\eta_{\w_1}-P_{\w_1}dr\wedge d\eta_{\w_1}\\
    &= Q'_{\w_1}dr\wedge\xi_{\w_1} +Q_{\w_1} \bar{d}\iota^*\xi_{\w_1}-P_{\w_1}dr\wedge\bar{d}\iota^*\eta_{\w_1}\\
    &=dr\wedge\left(Q'_{\w_1}\xi_{\w_1}-P_{\w_1}\bar{d}\iota^*\eta_{\w_1}\right)+Q_{\w_1}\bar{d}\iota^*\xi_{\w_1}.
\end{align*}
Hence, we get \begin{align*}
    \nu\lrcorner d{\w_1}=-\partial_r\lrcorner d{\w_1}=-Q'_{\w_1}(1)\xi_{\w_1} +P_{\w_1}(1)\bar{d}\iota^*\eta_{\w_1}.
\end{align*}
Next, we compute $\nu\lrcorner d(\nabla_x {\w_1})$. 
\begin{align*}
\nabla_x{\w_1} &= r\cdot\nabla_{\partial_r}{\w_1} \\
&\hspace{-3mm}\overset{\eqref{eq:calculConditionsdebord}}{=} -r\cdot\Bigg[ \xi_{\w_1}\left(-Q'_{\w_1}(r)+\frac{p}{r}Q_{\w_1}(r)\right) \\
&\phantom{={}} \qquad + dr\wedge\eta_{\w_1}\left(-P'_{\w_1}(r)+P_{\w_1}(r)\frac{p-1}{r}\right) \Bigg] \\
&= \xi_{\w_1}(rQ'_{\w_1}(r)-pQ_{\w_1}(r)) \\
&\quad - dr\wedge\eta_{\w_1}(-rP'_{\w_1}(r)+P_{\w_1}(r)(p-1)).
\end{align*}
We still need to find $d(\nabla_x {\w_1}).$ 
\begin{align*}
    d(\nabla_x {\w_1})&=\left((rQ'_{\w_1}(r))'-pQ'_{\w_1}(r)\right)dr\wedge\xi_{\w_1}+(rQ'_{\w_1}(r)-pQ_{\w_1}(r))d\xi_{\w_1}\\
    &+dr\wedge d\eta_{\w_1}(-rP'_{\w_1}(r)+P_{\w_1}(r)(p-1)).
\end{align*}
Thus, we get $$\begin{aligned} \partial_r\lrcorner d(\nabla_x {\w_1}) &= \left(-pQ'_{\w_1}(1)+Q'_{\w_1}(1)+Q''_{\w_1}(1)\right)\xi_{\w_1} \\ &\quad + \left(-P'_{\w_1}(1)+P_{\w_1}(1)(p-1)\right)\bar{d}\iota^*\eta_{\w_1}. \end{aligned}$$
Therefore, we conclude that: \begin{align*}
    \nu \lrcorner d\Delta\w+\ell\iota^*\w=0&\iff 2nQ'_{\w_1}(1)\xi_{\w_1}-2nP_{\w_1}(1)\bar{d}\iota^*\eta_{\w_1}\\
    &\qquad \quad+4\xi_{\w_1}\left((1-p)Q'_{\w_1}(1)+Q''_{\w_1}(1)\right)\nonumber\\\nonumber
   & \qquad \quad+4\left(-P'_{\w_1}(1)+P_{\w_1}(1)(p-1)\right)\bar{d}\iota^*\eta_{\w_1}\\
   &\qquad \quad-2\ell Q_{\w_2}(1)\iota^*\xi_{\w_2}=0\nonumber\\ \nonumber
    &\iff \xi_{\w_1}\left((4-4p+2n)Q'_{\w_1}(1)+4Q''_{\w_1}(1)\right)\\
    &\qquad \quad +\bar{d}\iota^*\eta_{\w_1}\left((-2n+4p-4)P_{\w_1}(1)-4P'_{\w_1}(1)\right)\\
    & \qquad \quad-2\ell Q_{\w_2}(1)\iota^*\xi_{\w_2}=0.
\end{align*}
Finally, using \eqref{eq:condition3}, we get the following boundary condition: \begin{equation}\label{eq:condition4}
\begin{split}
     \nu \lrcorner d\Delta\w+\ell \iota^*\w=0 &\iff \xi_{\w_1}\left((4-4p+2n)Q'_{\w_1}(1)+4Q''_{\w_1}(1)\right) \\
     &\quad \qquad -2\ell Q_{\w_2}(1)\xi_{\w_2}=0.
\end{split}
\end{equation}
\end{itemize} This concludes the proof of \propref{prop:conditions}. \end{proof}

\begin{lemma}\label{lem:xinonul}
The $p$-form $\xi_{\w_2}$ is non-zero.
\end{lemma}
\begin{proof}
We suppose by contradiction that $\xi_{\w_2}=0$, then $\eta_{\w_2}\neq 0$ and $P_{\w_2}\neq 0$ since $\w_2\neq0$ by \lemref{lm:wnonnul}. By \eqref{eq:AWBW} we have that $\bar{d}\eta_{\w_2}=0$ and $\bar{\Delta}\eta_{\w_2}=\lambda\eta_{\w_2}$ for some $\lambda\in\mathbb{R}$. In addition, $P''_{\w_2}+\frac{n-2p+1}{r}P'_{\w_2}-\frac{n-2p+1+\lambda}{r^2}P_{\w_2}=0.$ Therefore, by \thref{thm:vplaplacien}, there exist $k,p\in\mathbb{N}$ such that $\lambda=\boldsymbol{\lambda_{k,p-1}^{\Ex}}=\boldsymbol{\lambda_{k,p-2}^{\Coex}}$ i.e., $\boldsymbol{\lambda_{k,p-1}^{\Ex}}=(k+p-1)(k+n-p+1)$. By \lemref{lem:PetQpolynomes}, we get $P_{\w_2}(r)=r^{k+p}.$ This contradicts \eqref{eq:condition1}, and therefore, $\xi_{\w_2}\neq 0$.
\end{proof}
In the final step of the proof of \thref{thm:spectreNBS}, we consider the following Proposition: 
\begin{prop}\label{prop:propspectre}
Let $\w\in\Omega^p(\mathbb{B}^n)$ be such that $\w(x)=-2\w_2(x)+(\lvert x\rvert^2-1)\w_1(x)$ where $\omega_1$ and $\omega_2$ are harmonic and given as in \eqref{eq:W=P+Q}. Then $\w$ is an eigenform for the NBS \eqref{eq:NBSF3preuve} if and only if $\xi_{\w_2}$ is an exact or co-exact eigenform of the Laplacian on $\mathbb{S}^{n-1}$. The associated eigenvalues are given by
\begin{equation*}
    \ell^{\Coex}_{k,p} = (k+p+1)^2(n+2(k+1))
\end{equation*}
if $\xi_{\w_2}$ is a co-exact $p$-form associated to $\boldsymbol{\lambda^{\Coex}_{k,p}} = (k + p + 1)(n - p + k-1)$, and
\begin{equation*}
    \ell^{\Ex}_{k,p} = \frac{(2k+n+2)^2(k+p)^2(2k+n+4)}{4k^2+4kn+4k+n^2+4n-4p}
\end{equation*}
if $\xi_{\w_2}$ is an exact $p$-form associated to $\boldsymbol{\lambda^{\Ex}_{k,p}} = (k + p)(n - p + k )$. The respective multiplicities are:
\[\dim(\ker(\bar{\Delta}_p - \boldsymbol{\lambda^{\Coex}_{k,p}}))\qquad
\text{and}
    \qquad \dim(\ker(\bar{\Delta}_p - \boldsymbol{\lambda^{\Ex}_{k,p}})).\]
\end{prop}
\begin{proof}
Let $\w \in \Omega^p(\mathbb{B}^n)$ satisfy the hypotheses of the proposition. We consider the following cases: 
\begin{itemize}
    \item ${\underline{\text{If}\ P_{\w_2}\bar{\delta}\eta_{\w_2}\neq0\text{:}}}$
    
  Based on \lemref{lem:PetQpolynomes}, equation $\bar{\delta}(E_{\radd})$ gives us $P_{\w_2}(r)=r^{k'+p}$. However, the condition \eqref{eq:condition1} implies that either $P_{\w_2}(1)=0$ or $\eta_{\w_2}=0$. Since $P_{\w_2}(1)=1$, it necessarily follows that $\eta_{\w_2}=0$, which is a contradiction.

 \item ${\underline{\text{If}\ P_{\w_2}\bar{\delta}\eta_{\w_2}=0, \ \bar{d}\xi_{\w_2}\neq0\text{:}}}$
 
        Equation $\bar{\delta}(E_{\radd})$ in \eqref{eq:DAWDQW} becomes trivial. In addition, based on\\ \lemref{lem:PetQpolynomes}, equation $\bar{d}(E_{\tann})$ gives $Q_{\w_2}(r)=r^{k+p+1}$. Consequently, $(E_{\tann})$ becomes:\[r^{k+p-1}\bar{\Delta}\xi_{\w_2}=\Big((k+p+1)(k+p)+(n-2p-1)(k+p+1)\Big)r^{k+p-1}\xi_{\w_2}+\frac{2P_{\w_2}}{r}\bar{d}\eta_{\w_2}.\] Therefore, \[\underbrace{\bar{\Delta}\xi_{\w_2}-(k+p+1)(n+k-p-1)\xi_{\w_2}}_{\text{independent of $r$}}=2r^{-k-p}P_{\w_2}\bar{d}\eta_{\w_2}.\] The right-hand side is independent of $r$ if and only if $\bar{d}\eta_{\w_2}=0$ or $P_{\w_2}(r)=r^{k+p}$ up to rescaling $\eta_{\w_2}$, which implies by \eqref{eq:condition1} that $\eta_{\w_2}=0$. Since $P_{\w_2}\bar{d}\eta_{\w_2}=0$, then for $p\geq1$, we have $P_{\w_2}\eta_{\w_2}=0$ and $\xi_{\w_2}$ is an eigenform of the spherical Laplacian, co-closed since $\bar{d}\xi_{\w_2}=0$. By condition \eqref{eq:condition2}, we have \begin{equation}\label{eq:q'xi=qxi}
             \underbrace{Q'_{\w_2}(1)}_{k+p+1}\xi_{\w_2}=Q_{\w_1}(1)\xi_{\w_1}. \end{equation} We can normalize $\xi_{\w_1}$ such that $\xi_{\w_1}=(k+p+1)\xi_{\w_2},$ we get $Q_{\w_1}(1)=1$, thus we write $Q_{\w_1}(r)$ in polynomial form $Q_{\w_1}(r)=r^{k+p+1}$. Replacing in condition \eqref{eq:condition4}, we get \[((4-4p+2n)(k+p+1)+4(k+p)(k+p+1))\xi_{\w_1}-2\ell\frac{\xi_{\w_1}}{k+p+1}=0.\]
     Since $\xi_{\w_1}\neq0$ \lemref{lem:xinonul}, we finally get \begin{equation}\label{eq:valeurPorpreL1}
         \ell^{\Coex}_{k,p}=(k+p+1)^2(n+2(k+1)).
     \end{equation}Note that $\ell^{\Coex}_{k,p}>0$. This expression coincides with the one found in \cite{WangXia} for $p = 0$. The multiplicity of this eigenvalue is equal to the multiplicity of $\xi_{\w_2}$ as a co-exact eigenform of the Laplacian on the sphere. Thus, it corresponds to the multiplicity of the eigenvalue $\boldsymbol{\lambda^{\Coex}_{k,p}}$, which is, by \thref{thm:vplaplacienmulti}, given by $$\dim(\ker(\bar{\Delta}_p-\boldsymbol{\lambda^{\Coex}_{k,p}}))=\frac{(n + k-3)! (n + 2k-2)}{p!  k! (n - p - 4)!  (n + k - p-3)(k + p + 1)}.$$ 
    From condition \eqref{eq:condition3}, we have $\eta_{\w_1}=0$ or $(2n-4p+4)P_{\w_1}(1)+4P'_{\w_1}(1)=0$. When $\eta_{\w_1}\neq0$, since $\bar{\delta}\xi_{\w_2}=0$, we get that $\bar{\delta}\xi_{\w_1}=0$ by \eqref{eq:q'xi=qxi}, and by \eqref{eq:AWBW} and \lemref{lem:PetQpolynomes} that $P_{\w_1}(r)=c r^{k+p}$. Replacing in \eqref{eq:condition3}, we get $(2n+4k+4)c=0$, then $c=0$, therefore, $P_{\w_1}=0$. Therefore, we get eventually \begin{equation}\label{eq:wrtheta}
        \w(r,\theta)=\w_{k,p}^{\Coex}(r,\theta)=r^{k+p+1}\xi_{\w_2}(\theta)(-2+(r^2-1)(k+p+1))
    \end{equation}
where $\xi_{\w_2}$ is a co-exact eigenform of the spherical Laplacian with eigenvalue $\boldsymbol{\lambda^{\Coex}_{k,p}} = (k + p + 1)(n - p + k)$.\\ We let the reader check that conversely, any $\w_{k,p}^{\Coex}$ of the form \eqref{eq:wrtheta} is an eigenform of NBS whose eigenvalue is given by \eqref{eq:valeurPorpreL1}.
    \item ${\underline{\text{If}\ P_{\w_2}\bar{\delta}\eta_{\w_2}=0, \ \bar{d}\xi_{\w_2}=0\text{:}}}$
      \begin{itemize}
        \item ${\underline{\text{If}\ P_{\w_2}\neq0\text{:}}}$
      \begin{itemize}
       \item[$\star$] ${\underline{\text{If}\ \bar{\delta}\eta_{\w_2}=0,\ Q_{\w_2}=0\text{:}}}$

Equation $(E_{\tann})$ \eqref{eq:AWBW} gives us that $P_{\w_2}\bar{d}\eta_{\w_2}=0.$ However, since $P_{\w_2}\neq0,$ (otherwise $\w_2=0$, see \lemref{lm:wnonnul}) then $\bar{d}\eta_{\w_2}=0$. In addition, having $\bar{\delta}\eta_{\w_2}=0$, we get $\eta_{\w_2}=0$ for $p\geq1.$ Thus ${\w_2}=0,$ which yields a contradiction.
       
          \item[$\star$] ${\underline{\text{If}\ \bar{\delta}\eta_{\w_2}=0,\ Q_{\w_2}\neq0\text{:}}}$
          
Applying $\bar{d}$ to equation $(E_{\radd})$ of \eqref{eq:AWBW}, we obtain:
\begin{equation}\label{eq:eqcas4}
    \frac{P_{\w_2}}{r^2}\bar{\Delta}\bar{d}\eta_{\w_2}=(P''_{\w_2}-\frac{n-2p+1}{r^2}P_{\w_2}+\frac{n-2p+1}{r}P'_{\w_2})\bar{d}\eta_{\w_2}+\frac{2Q_{\w_2}}{r^3}\bar{\Delta}\xi_{\w_2}.
\end{equation}
Substituting equation $(E_{\tann})$ of \eqref{eq:AWBW} into \eqref{eq:eqcas4}, we get:
\begin{equation*}
\begin{aligned}
\frac{P_{\w_2}}{r^2}\bar{\Delta}\bar{d}\eta_{\w_2}= &\left(P''_{\w_2}-\frac{n-2p+1}{r^2}P_{\w_2}+\frac{n-2p+1}{r}P'_{\w_2}\right)\bar{d}\eta_{\w_2}\\
&+ \frac{2}{r}\left(Q''_{\w_2}+\frac{n-2p-1}{r}Q'_{\w_2}\right)\xi_{\w_2}+ \frac{4P_{\w_2}}{r^2}\bar{d}\eta_{\w_2}.
\end{aligned}
\end{equation*}
Isolating $\bar{\Delta}\bar{d}\eta_{\w_2}$, we finally obtain: \begin{equation}\label{eq:derniereeqcas4}
\begin{aligned}
\bar{\Delta}\bar{d}\eta_{\w_2}= &\frac{1}{P_{\w_2}}\left(r^2P''_{\w_2}-(n-2p-3)P_{\w_2}+r(n-2p+1)P'_{\w_2}\right)\bar{d}\eta_{\w_2}\\
&+ \frac{2}{P_{\w_2}}\left(rQ''_{\w_2}+(n-2p-1)Q'_{\w_2}\right)\xi_{\w_2}.
\end{aligned}
\end{equation}
We write $(E_{\tann})$ of \eqref{eq:AWBW} and \eqref{eq:derniereeqcas4} in the following form: 
\begin{equation}\label{eq:systemematrice}
\resizebox{\linewidth}{!}{%
$\begin{pmatrix}
\bar{\Delta}\xi_{\w_2}\\[4mm]
\bar{\Delta}\bar{d}\eta_{\w_2}
\end{pmatrix}
=
\begin{pmatrix}
\dfrac{r}{Q_{\w_2}}\!\left(r Q''_{\w_2}+ (n-2p-1)Q'_{\w_2}\right)
&
\dfrac{2r}{Q_{\w_2}} P_{\w_2}
\\[1.5em]
\dfrac{2}{P_{\w_2}}\!\left(r Q''_{\w_2}+ (n-2p-1)Q'_{\w_2}\right)
&
\dfrac{1}{P_{\w_2}}\!\left(r^2 P''_{\w_2}+ r(n-2p+1)P'_{\w_2}- (n-2p-3)P_{\w_2}\right)
\end{pmatrix}
\cdot
\begin{pmatrix}
\xi_{\w_2}\\[6mm]
\bar{d}\eta_{\w_2}
\end{pmatrix}$%
}
\end{equation}\vspace*{\baselineskip}
\begin{itemize}[label=$\blacktriangleright$]
    \item ${\underline{\text{If}\ \xi_{\w_2}\ \text{and}\ \bar{d}\eta_{\w_2}\ \text{are linearly independent as forms on}\  \mathbb{S}^{n-1}:}}$
    
The equations of the system are of the form: 
\begin{align*}
\bar{\Delta}\xi_{\w_2}= a(r)\xi_{\w_2}+ b(r)\bar{d}\eta_{\w_2}\quad \text{and} \quad \bar{\Delta}\bar{d}\eta_{\w_2}= c(r)\xi_{\w_2}+ d(r)\bar{d}\eta_{\w_2}.
\end{align*} Since $\xi_{\w_2}$ and $\eta_{\w_2}$ are independent of $r$, differentiating the equations of that system yields
\begin{align*}
a'(r)\xi_{\w_2}+ b'(r)\bar{d}\eta_{\w_2}= 0 \quad \text{and} \quad c'(r)\xi_{\w_2}+ d'(r)\bar{d}\eta_{\w_2}= 0.
\end{align*} Since $\xi_{\w_2}$ and $\bar{d}\eta_{\w_2}$ are linearly independent, it follows that $a'(r),$ $b'(r),$ $c'(r)$, and $d'(r)$ are zero. Consequently, $a(r), b(r), c(r)$, and $d(r)$ are constants. Moreover, they are positive since $Q_{\w_2}$ and $P_{\w_2}$ are powers of $r$ according to the equations of the system \eqref{eq:systemematrice} above. However, by condition \eqref{eq:condition1}, we must have either $P_{\w_2}(1)=0$ or $\eta_{\w_2}=0$. The first case is impossible, while the second case yields a contradiction since $\bar{d}\eta_{\w_2}= \xi_{\w_2}\neq 0$.
    \item ${\underline{\text{If}\ \xi_{\w_2}=t\bar{d}\eta_{\w_2}\text{:}}}$

    We can assume $t=1$. The system  \eqref{eq:systemematrice} becomes: \begin{equation}\label{eq:systemenonmatrice}
\begin{cases}
\bar{\Delta}\xi_{\w_2}
=
\dfrac{r}{Q_{\w_2}}\Bigl(r Q''_{\w_2}+ (n-2p-1) Q'_{\w_2}+2 P_{\w_2}\Bigr) \xi_{\w_2}\\[4mm]
\begin{aligned}
\bar{\Delta}\xi_{\w_2}
=
\Biggl( &\dfrac{2}{P_{\w_2}}\Bigl(r Q''_{\w_2}+ (n-2p-1) Q'_{\w_2}\Bigr) \\
&\hspace{-0.8cm} + \dfrac{1}{P_{\w_2}} \Bigl(r^2 P''_{\w_2} + r (n-2p+1) P'_{\w_2}- (n-2p-3) P_{\w_2}\Bigr)\Biggr)\, \xi_{\w_2}.
\end{aligned}
\end{cases}
\end{equation}Therefore, $\xi_{\w_2}=\bar{d}\eta_{\w_2}$ is an exact eigenform of the spherical Laplacian associated with eigenvalue $\boldsymbol{\lambda^{\Ex}_{k,p}} = (k + p)(n - p + k )$ for some $p\in\mathbb{N}$ by \thref{thm:vplaplacien}. We seek solutions of the form $Q_{\w_2}(r)=ar^{\alpha}$ and $P_{\w_2}(r)=br^\beta$. Substituting these into the first equation of \eqref{eq:systemenonmatrice}, we obtain:
$$\bar{\Delta}\xi_{\w_2}=\underbrace{\Big(\alpha(\alpha-1)+\alpha(n-2p-1) +\frac{2br^\beta}{ar^{\alpha-1}}\Big)}_{\boldsymbol{\lambda_{k,p}^{\Ex}}}\xi_{\w_2}.$$
Therefore, we obtain the following relation between $a$ and $b$: \begin{equation}\label{eq:b/a}
  b=a\cdot\frac{\boldsymbol{\lambda_{k,p}^{\Ex}}-\alpha(\alpha-1)-\alpha(n-2p-1)}{2}.
\end{equation}Substituting the expressions for $P_{\w_2}$ and $Q_{\w_2}$ into the second equation of \eqref{eq:systemenonmatrice}, we have: \begin{equation*}
\begin{aligned}
b\cdot\boldsymbol{\lambda_{k,p}^{\Ex}} &= \Biggl(2a\big(\alpha(\alpha-1)+\alpha(n-2p-1)\big) \\
&\hspace{1cm} + \beta b(\beta-1)-b(n-2p-3)+\beta b(n-2p+1)\Biggr).
\end{aligned}
\end{equation*}Using the relation between $a$ and $b$ \eqref{eq:b/a}, we get: 
\begin{align*}
&\frac{
4\alpha\bigl((\alpha-1)+(n-2p-1)\bigr)
}{
\boldsymbol{\lambda_{k,p}^{\Ex}}-\alpha(\alpha-1)-\alpha(n-2p-1)
}
+(\alpha-1)(\alpha-2)\\
&-(n-2p-3)
+(\alpha-1)(n-2p+1)
= \boldsymbol{\lambda_{k,p}^{\Ex}}
\\[1em]
\Longleftrightarrow\quad
&
4\alpha\bigl((\alpha-1)+(n-2p-1)\bigr)
\\
&\quad
+\Bigl(
\boldsymbol{\lambda_{k,p}^{\Ex}}-\alpha(\alpha-1)-\alpha(n-2p-1)
\Bigr)\\
&\quad  \times \Bigl(
(\alpha-1)(\alpha-2)
-(n-2p-3)\\
&\qquad\qquad\qquad\qquad+(\alpha-1)(n-2p+1)
\Bigr)
\\
&=
\boldsymbol{\lambda_{k,p}^{\Ex}}\Bigl(
\boldsymbol{\lambda_{k,p}^{\Ex}}-\alpha(\alpha-1)-\alpha(n-2p-1)
\Bigr)\end{align*}
\begin{align*}
\Longleftrightarrow\quad
&
4\alpha(\alpha+n-2p-2)
\\
&\quad
+\Bigl(
\boldsymbol{\lambda_{k,p}^{\Ex}}-\alpha(\alpha-1)-\alpha(n-2p-1)
\Bigr)\\
&\quad \times
\Bigl(
(\alpha-1)(\alpha-2)
-(n-2p-3)\\
&\quad
+(\alpha-1)(n-2p+1)
-\boldsymbol{\lambda_{k,p}^{\Ex}}
\Bigr)
\\ &\quad =0\\[1em]
\Longleftrightarrow\quad
&
4\alpha(\alpha+n-2p-2)
\\
&\quad
+\Bigl(
\boldsymbol{\lambda_{k,p}^{\Ex}}-\alpha(\alpha+n-2p-2)
\Bigr)\\
&\quad \times
\Bigl(
(\alpha-1)(\alpha-2+n-2p+1)
\\
&\qquad\qquad\qquad\qquad-(n-2p-3)
-\boldsymbol{\lambda_{k,p}^{\Ex}}
\Bigr)
=0
\end{align*}
\begin{align*}
\Longleftrightarrow\quad
&
4\alpha(\alpha+n-2p-2)
\\
&\quad
+\Bigl(
\boldsymbol{\lambda_{k,p}^{\Ex}}-\alpha(\alpha+n-2p-2)
\Bigr)\\
&\quad \times
\Bigl(
(\alpha-1)(\alpha+n-2p-1)
\\
&\qquad\qquad\qquad\qquad-(n-2p-3)
-\boldsymbol{\lambda_{k,p}^{\Ex}}
\Bigr)
=0.
\end{align*}
We obtain the following fourth-degree equation in $\alpha$ to solve:  \[
\begin{aligned}
4\alpha&(\alpha+n-2p-2) + \Bigl(\boldsymbol{\lambda_{k,p}^{\Ex}} -\alpha(\alpha+n-2p-2)\Bigr) \\
&\times \Bigl((\alpha-1)(\alpha+n-2p-1)-(n-2p-3)-\boldsymbol{\lambda_{k,p}^{\Ex}}\Bigr) = 0.
\end{aligned}
\]  
We obtain the following solutions:\begin{align*}
\alpha_1 &= k+p          & &; & \alpha_2 &= k+p+2 \\
\alpha_3 &= -k-n+p        & &; & \alpha_4 &= -k-n+p+2.
\end{align*}We reject the roots $\alpha_3$ and $\alpha_4$ since they are negative. Thus, we have:
\begin{align*}
\text{for } \alpha_1 = k+p: \qquad &
    \left\{
    \hspace{-0.3em}
    \begin{array}{r@{\quad}c@{\quad}l}
        Q_{\w_2}(r) &=& ar^{k+p} \\
        P_{\w_2}(r) &=& br^{k+p-1}
    \end{array}
    \right.
\\[1ex]
\text{and for } \alpha_2 = k+p+2: \qquad &
\left\{
    \hspace{-0.3em}
    \begin{array}{r@{\quad}c@{\quad}l}
        Q_{\w_2}(r) &=& ar^{k+p+2} \\
        P_{\w_2}(r) &=& br^{k+p+1}.
    \end{array}
    \right.
\end{align*}  Condition \eqref{eq:condition1} gives $P_{\w_2}(1)=0$ or $\eta_{\w_2}=0$. Having $P_{\w_2}(1)=0$ is impossible since it implies that $b=0$, which yields by \eqref{eq:b/a} that $a=0$ because $\lambda_{k,p}^{\Ex} - \alpha(\alpha-1) - \alpha(n-2p-1) \neq 0$. Actually, this expression is equal to $k+p$ for $\alpha_1 = k+p$ and to $-(k+n-p)$ for $\alpha_2 = k+p+2$. Having $\eta_{\w_2}=0$ gives us that $\xi_{\w_2}=0$ which is impossible by \lemref{lem:xinonul}. We then take the linear combination \begin{equation}\label{eq:star4}
    P_{\w_2}(r)=r^{k+p-1}-r^{k+p+1}
\end{equation} which satisfies condition \eqref{eq:condition1}. Then, for $b=1$, equation \eqref{eq:b/a} gives the coefficient of $Q_{\w_2}(r)$:
\[a=\frac{2}{\boldsymbol{\lambda_{k,p}^{\Ex}}-\alpha(\alpha-1)-\alpha(n-2p-1)}.\] Therefore, we write $Q_{\w_2}(r)$ as follows: \[
\begin{aligned}
Q_{\w_2}(r) = &\frac{2}{\boldsymbol{\lambda_{k,p}^{\Ex}}-\alpha_1(\alpha_1-1)-\alpha_1(n-2p-1)}\,r^{\alpha_1} \\
&- \frac{2}{\boldsymbol{\lambda_{k,p}^{\Ex}}-\alpha_2(\alpha_2-1)-\alpha_2(n-2p-1)}\,r^{\alpha_2}.
\end{aligned}
\] 
Simplifying the above expression for $Q_{\w_2}(r)$, we finally get: \begin{align}\label{eq:star1}
Q_{\w_2}(r)
&=
\frac{2}{\boldsymbol{\lambda_{k,p}^{\Ex}}-(k+p)(k+p-1)-(k+p)(n-2p-1)}\, r^{k+p}
\nonumber \\ \nonumber
&\quad
-\frac{2}{\boldsymbol{\lambda_{k,p}^{\Ex}}-(k+p+2)(k+p+1)-(k+p+2)(n-2p-1)}\, r^{k+p+2} \nonumber\\
&=\frac{r^{k+p}}{k+p}+\frac{r^{k+p+2}}{k+n-p}.
\end{align}
 Based on condition \eqref{eq:condition2}, for $Q_{\w_1}(1)=1,$ we have: \begin{equation}\label{eq:xiH=xiW}
        \xi_{\w_1}=Q'_{\w_2}(1)\xi_{\w_2}=(1+\frac{k+p+2}{k+n-p})\underbrace{\xi_{\w_2}}_{\bar{d}\eta_{\w_2}}. \end{equation}
        
        Replacing \eqref{eq:xiH=xiW} in $(E_{\tann})$ in \eqref{eq:AWBW}, we obtain:\[
\begin{aligned}
\frac{Q_{\w_1}(r)}{r^2}\boldsymbol{\lambda_{k,p}^{\Ex}}
\Big(1+\frac{k+p+2}{k+n-p}\Big)\bar d\eta_{\w_2}
&=
Q_{\w_1}''(r)
\Big(1+\frac{k+p+2}{k+n-p}\Big)\bar d\eta_{\w_2}\\
&\quad+
\frac{n-2p-1}{r}Q_{\w_1}'(r)\\
&\quad\times
\Big(1+\frac{k+p+2}{k+n-p}\Big)\bar d\eta_{\w_2}\\
&\quad+
\frac{2P_{\w_1}(r)}{r}\bar d\eta_{\w_1}.
\end{aligned}
\] After simplification, we get: \[
\begin{aligned}
\Big(
 -Q''_{\w_1}(r)
 -\frac{n-2p-1}{r}Q'_{\w_1}(r)
  +\frac{Q_{\w_1}(r)}{r^2}\boldsymbol{\lambda_{k,p}^{\Ex}}
\Big)\\
\times\Big(
  1+\frac{k+p+2}{k+n-p}
\Big)
\bar{d}\eta_{\w_2}\\=
\frac{2P_{\w_1}}{r}\bar{d}\eta_{\w_1}.
\end{aligned}
\] Therefore, 
\[\Big(
  Q''_{\w_1}(r)
  +\frac{n-2p-1}{r}Q'_{\w_1}(r)
  -\frac{Q_{\w_1}(r)}{r^2}\boldsymbol{\lambda_{k,p}^{\Ex}}
\Big)\xi_{\w_1}
+\frac{2P_{\w_1}(r)}{r}\bar{d}\eta_{\w_1}=0.\]
Applying $\bar{d}$ and $\bar{d}\bar{\delta}$ to $(E_{\radd})$ in \eqref{eq:AWBW} and \eqref{eq:xiH=xiW}, respectively, we get on the one hand: \begin{equation}\label{eq:dBH}
\begin{aligned}
\bar{d}(E_{\radd}) : \frac{P_{\w_1}(r)}{r^2}\bar{\Delta}\bar{d}\eta_{\w_1} = \,&P''_{\w_1}(r)\bar{d}\eta_{\w_1} - \frac{n-2p+1}{r^2}P_{\w_1}(r)\bar{d}\eta_{\w_1} \\
&+ \frac{n-2p+1}{r}P'_{\w_1}(r)\bar{d}\eta_{\w_1}\\
&+ \frac{2Q_{\w_1}(r)}{r^2}\bar{d}\bar{\delta}\xi_{\w_1}
\end{aligned}
\end{equation}  and on the other hand:\begin{align}\label{eq:ddeltaxi=xi}
   \bar{d}\bar{\delta}\xi_{\w_1}&=\Big(1+\frac{k+p+2}{k+n-p}\Big)\bar{d}\bar{\delta}\xi_{\w_2}\nonumber\\ \nonumber
   &=\Big(1+\frac{k+p+2}{k+n-p}\Big)\Big(\bar{\Delta}\xi_{\w_2}-\underbrace{\bar{\delta}\bar{d}\xi_{\w_2}}_{0}\Big)\\\nonumber
   &=\Big(1+\frac{k+p+2}{k+n-p}\Big)\bar{\Delta}\xi_{\w_2}\\\nonumber
   &=\Big(1+\frac{k+p+2}{k+n-p}\Big)\boldsymbol{\lambda_{k,p}^{\Ex}}\xi_{\w_2}\\
   &=\boldsymbol{\lambda_{k,p}^{\Ex}}\xi_{\w_1}.
\end{align} Substituting \eqref{eq:ddeltaxi=xi} into \eqref{eq:dBH}, we thus obtain: \[
\begin{aligned}
\xi_{\w_1} + \frac{r^3}{2Q_{\w_1}(r)\boldsymbol{\lambda_{k,p}^{\Ex}}} \Biggl( &P''_{\w_1}(r) + \frac{n-2p+1}{r}P'_{\w_1}(r) \\
&- \frac{\boldsymbol{\lambda_{k,p}^{\Ex}}+n-2p+1}{r^2}P_{\w_1}(r) \Biggr) \bar{d}\eta_{\w_1} = 0.
\end{aligned}
\] 
Therefore, we have $\xi_{\w_1}=c\ \bar{d}\eta_{\w_1}$ with $c$ a non-zero constant which we can normalize to be $c=1$. Plugging this into \eqref{eq:AWBW}, we obtain: \[
\begin{cases}
Q''_{\w_1}(r)
+\dfrac{n-2p-1}{r}Q'_{\w_1}(r)
-\dfrac{\boldsymbol{\lambda_{k,p}^{\Ex}}}{r^2}Q_{\w_1}(r)
+\dfrac{2}{r}P_{\w_1}(r)
=0
\\[1em]
\begin{aligned}
P''_{\w_1}(r)
&+\dfrac{n-2p+1}{r}P'_{\w_1}(r) \\
&-\biggl(
\dfrac{\boldsymbol{\lambda_{k,p}^{\Ex}}+n-2p+1}{r^2}
\biggr)P_{\w_1}(r)
+\dfrac{2\boldsymbol{\lambda_{k,p}^{\Ex}}}{r^3}Q_{\w_1}(r)
=0.
\end{aligned}
\end{cases}
\]  This is an Euler system, we look for solutions of the form:
\[
\begin{aligned}
P_{\w_1}(r) &= ar^m & \text{and} \qquad Q_{\w_1}(r) &= br^{m+1}.
\end{aligned}
\] 
The solutions we get are: \begin{align}
    P_{\w_1}(r) &= a
    \left(
    r^{k+p-1}
    -\frac{2k+n}{2k+n+4}\, r^{k+p+1}
    \right) 
    \label{eq:star2} \\[0.5em]  
    Q_{\w_1}(r) &= a
    \left(
    \frac{r^{k+p}}{k+p}
    +\frac{2k+n}{2k+n+4}
    \cdot
    \frac{r^{k+p+2}}{k+n-p}
    \right).
    \label{eq:star3}
\end{align}However, $Q_{\w_1}(1)=1,$ therefore, we have the following expression for $a$:\begin{equation}\label{eq:A/L}
a=\frac{(k+p)(2k+n+4)(k+n-p)}{(2k+n+4)(k+n-p)+(2k+n)(k+p)}.\end{equation} Using \eqref{eq:xiH=xiW} and \eqref{eq:A/L}, we obtain the following based on \eqref{eq:condition4}: \begin{align*}
&\frac{(k+p)(2k+n+4)(2k+n+2)}{(2k+n+4)(k+n-p)+(2k+n)(k+p)}\\
&\times \Bigg[
(4-4p+2n)\left(1+\frac{(2k+n)(k+p+2)}{(2k+n+4)(k+n-p)}\right) \\
&+4\left((k+p-1)+\frac{(2k+n)(k+p+2)(k+p+1)}{(2k+n+4)(k+n-p)}\right)
\Bigg] \\
&-2\ell^{\Ex}\left(\frac{1}{k+p}+\frac{1}{k+n-p}\right)=0.
\end{align*}After simplifying, we finally get:\begin{equation}\label{eq:valeurPorpreL2}
    \ell^{\Ex}_{k,p}=\frac{(2k+n+2)^2(k+p)^2(2k+n+4)}{4k^2+4kn+4k+n^2+4n-4p}.
\end{equation}
The multiplicity of this eigenvalue is that of $\xi_{\w_2}$ as an exact eigenform of the Laplacian on the sphere. Therefore, it corresponds to the multiplicity of the eigenvalue $\boldsymbol{\lambda_{k,p}^{\Ex}}=(k+p)(k+n-p)$ (see, \thref{thm:vplaplacienmulti}), which is: {\small$$ \dim(\ker(\bar{\Delta}_p-\boldsymbol{\lambda^{\Ex}_{k,p}}))=\frac{(n + k-3)! (n + 2k-2)}{(p - 1)!  k!  (n - p-3)!  (n + k - p - 2 )(k + p)}.$$} 
Recall that by \eqref{eq:xiH=xiW}, $\xi_{\w_1}$ is an eigenform of ${(\bar{\Delta}_p)}_{|\Ex}$ associated with $\boldsymbol{\lambda_{k,p}^{\Ex}}=(k+p)(k+n-p)$. Applying $\bar{\delta}$ to $(E_{\radd})$ in \eqref{eq:AWBW}, in the case where $P_{\w_1}\bar{\delta}\eta_{\w_1} \neq 0$, $P_{\w_1}(r)$ is written as a positive power of $r$. This contradicts condition \eqref{eq:condition3}, since if $\eta_{\omega_1} = 0$, we get $P_{\omega_1}\bar{\delta}\eta_{\omega_1} = 0$, in addition the equation for $P_{\omega_1}$ is not satisfied, as we obtain $2n + 4k + 4 \neq 0$. Therefore, we must have $P_{\w_1}\bar{\delta}\eta_{\w_1} = 0$, and since $P_{\w_1} \neq 0$, it follows that $\bar{\delta}\eta_{\w_1} = 0$. Since $\bar{d}\eta_{\w_1}$ is an eigenform and $\bar{\delta}\eta_{\w_1} = 0$, we have $\bar{\Delta}(\bar{d}\eta_{\w_1}) = \boldsymbol{\lambda^{\Ex}_{k,p}}\bar{d}\eta_{\w_1}$, which implies $\bar{d}(\bar{\Delta}\eta_{\w_1} - \boldsymbol{\lambda^{\Ex}_{k,p}}\eta_{\w_1}) = 0$. Thus, since $\bar{\Delta}\eta_{\w_1} - \boldsymbol{\lambda^{\Ex}_{k,p}}\eta_{\w_1}$ is closed, it can be written as $\bar{\Delta}\eta_{\w_1} - \boldsymbol{\lambda^{\Ex}_{k,p}}\eta_{\w_1} = \bar{d}\theta$. Applying $\bar{\delta}$ yields $\bar{d}\theta = 0$, then $\eta_{\w_1}$ is an eigenform for ${(\bar{\Delta}_p)}_{|\Ex}$. Similarly, we can show that $\eta_{\w_2}$ is also an eigenform for ${(\bar{\Delta}_p)}_{|\Ex}$. Since $\xi_{\w_i}=\bar{d}\eta_{\w_i}$ for $i\in\{1,2\}$, we get by \eqref{eq:xiH=xiW} that $\eta_{\w_1}=(1+\frac{k+p+2}{k+n-p})\eta_{\w_2}.$ Eventually, we have \begin{equation}\label{eq:wrtheta1}
\begin{split}
    \w(r,\theta) &= \w_{k,p}^{\Ex}(r,\theta) \\
    &= dr \wedge \eta_{\w_2}(\theta) ( -2P_{\w_2}(r) + (r^2-1)(1+\frac{k+p+2}{k+n-p})P_{\w_1}(r) ) \\
    &\quad + \xi_{\w_2}(\theta) \left( -2Q_{\w_2}(r) + (r^2-1)Q_{\w_1}(r) \right)
\end{split}
\end{equation} where $\xi_{\w_2}$ is an exact eigenform of the spherical Laplacian associated to $\boldsymbol{\lambda^{\Ex}_{k,p}} = (k + p )(n - p + k),$ and $\eta_{\w_2}$ is the unique co-exact eigenform of the Laplacian such that $\bar{d}\eta_{\w_2}=\xi_{\w_2}$, $P_{\w_2}(r)$ is given by \eqref{eq:star4}, $Q_{\w_2}(r)$ by \eqref{eq:star1}, $P_{\w_1}(r)$ by \eqref{eq:star2} and $Q_{\w_1}(r)$ by \eqref{eq:star3}.\\ Conversely, the reader can check that \eqref{eq:wrtheta1} defines an eigenform for NBS \eqref{eq:NBSF3preuve} associated to the eigenform given in \eqref{eq:valeurPorpreL2}. 
   \end{itemize}
\end{itemize}
   \item ${\underline{\text{If}\ P_{\w_2}=0\text{:}}}$
   
   Equations $(E_{\tann})$ and $(E_{\radd})$ in \eqref{eq:AWBW} give respectively, $Q_{\w_2}(r)=r^{k+p+1}$ and $Q_{\w_2}(r)\bar{\delta}\xi_{\w_2}=0.$ Since $Q_{\w_2}(r)\neq0$ for $r\neq0,$ we must have $\bar{\delta}\xi_{\w_2}=0$. Furthermore, because $\bar{d}\xi_{\w_2}=0,$ it follows that $\xi_{\w_2}=0$ for $p\geq1$. Thus, $\w_2=0$, which leads to a contradiction.
\end{itemize}
\end{itemize} This concludes the proof \propref{prop:propspectre}.    
\end{proof}

To finish the proof of \thref{thm:spectreNBS}, we must show the following Proposition:
\begin{prop}\label{prop:spectrentier}
 The eigenvalues given in \eqref{eq:valeurPorpreL1} and \eqref{eq:valeurPorpreL2} together with $\{0\}$ when $H_{\Abs}(M)\neq0$ constitute the entire non-zero spectrum of the NBS operator on the ball.
\end{prop} Before the proof, let us first define the space $\check{Z}$: \begin{align*}
    \check{Z} := \bigl\{ \w \in \Omega^p(M) \bigm|\ &\Delta^2\w = 0 \text{ on } M, \ \nu\lrcorner\w=0,\ \nu\lrcorner d\w=0,\ \nu\lrcorner \Delta\w=0\ \text{on}\ \partial M \\
    &\qquad\qquad\qquad\qquad\qquad\qquad\qquad\qquad\qquad\text{and } \w \perp_{L^2(\partial M)} H_{\Abs}(M) \bigr\}.
\end{align*}

\begin{proof}[Proof of \propref{prop:spectrentier}.]
It follows from the proof of \thref{thm:TheoremeGeneralNBSF3PourLaPreuve} (see, \cite{Article1}, Section 3) that any eigenform of NBS is either in $H_{\Abs}(\mathbb{B}^n)$ (if it is in the kernel of NBS) or in $\check{Z}$.
\begin{lemma}\label{lm:iotainj}
The map \begin{align*}
\iota^* : V &\longrightarrow H^{\frac{3}{2}}(\partial M) \\
\omega &\longmapsto \iota^*\omega
\end{align*} is injective.
\end{lemma}

\begin{proof}
  Let $\omega \in \check{Z}$ such that $\iota^* \omega = 0$. We have then $\omega_{|\partial M} = 0$. Moreover, by integration by parts \eqref{eq:IPP2}, we have that $\Delta \omega = 0$. Therefore by \cite{ColetteAnne} $\omega = 0$.
\end{proof}
Note that for any eigenform of the NBS given by $\eqref{eq:wrtheta}$ or $\eqref{eq:wrtheta1}$, we have $\iota^* \omega = -2 \xi_{\omega_2}$. Ranging among all possible such $\xi_{\w_2}$'s, we get that $\iota^*\w$ can be any eigenform of the Laplacian on $\mathbb{S}^{n-1}$. Recall that these eigenforms generate a dense vector space in $L^2(\mathbb{S}^{n-1})$. Therefore, if $\omega$ is an $\ell$-eigenform for NBS, its trace can be decomposed as:
\[
\iota^*\omega = \sum_{i=1}^\infty a_i \xi_i, \quad a_i \in \mathbb{R}, \quad a_i = (\iota^*\omega, \xi_i)_{L^2(\partial M)}.
\]

We know that for all $i\geq1$, $\xi_i = \iota^*\omega_i$, where $\omega_i$ is an $\ell_i$-eigenform of NBS.

Assume by contradiction that $\ell \neq \ell_i$ for all $i$. Then, using an integration by parts \eqref{eq:IPP2}, we have:
\[
 (\Delta\omega, \Delta\omega_i)_{L^2(M)} = \ell(\iota^*\omega, \iota^*\omega_i)_{L^2(\partial M)}=\ell a_i.
\] Moreover, it follows from the spectral study of NBS (see, \cite{Article1}, Section 3) that if $\ell\neq\ell_i,$ then $(\Delta\omega, \Delta\omega_i)_{L^2(M)}=0.$
This implies that $a_i = 0$, hence $\iota^*\omega = 0$, which gives $\omega = 0$ since $\iota^*$ is injective by \lemref{lm:iotainj}. Therefore, there cannot be any other eigenform.
\end{proof}
This completes the proof of \thref{thm:spectreNBS}.
\end{proof}

\begin{proof}[Proof of \thref{thm:spectreNBScasparticulier}.]
Assume $n \ge 2$ and $p = 0$ i.e., we consider the scalar NBS on $\mathbb{B}^n$. The spectrum of the Laplacian on scalar functions on $\mathbb{S}^{n-1}$ is given in \thref{thm:vplaplacienFCT}.

\begin{itemize}
    \item For $k = 0$, $\boldsymbol{\lambda_{0,0}} = 0$ and the constant map $f(x) \equiv 1$ is obviously a solution to NBS with eigenvalue $\ell_{0,0} = 0$. Note that $\iota^* f = f_{|\mathbb{S}^{n-1}} \equiv 1$.
    \item For $k \ge 1$, let $\xi_k$ be an eigenfunction of the Laplacian on $\mathbb{S}^{n-1}$ associated to $\boldsymbol{\lambda_k}$. Then the proof of \eqref{eq:wrtheta} shows that
    \[
    f_k(r, \theta) = r^{k+1} \xi_k(\theta) (-2 + (r^2-1)(k+1)) 
    \]
    is an eigenfunction for NBS associated to the eigenvalue $\ell_{k,0} = k^2 (n+2k)$. 
\end{itemize}

Since for each eigenfunction $\xi_k$ of the Laplacian on $\mathbb{S}^{n-1}$ we have constructed an associated eigenfunction $f_k$ for NBS such that ${(f_k)}_{|\mathbb{S}^{n-1}} = \xi_k$, repeating the density argument of \propref{prop:spectrentier} yields \thref{thm:spectreNBScasparticulier} for $p = 0$.\medskip 

Assume now $n \ge 2$ and $p = n-1$. The spectrum of $\bar{\Delta}$ on $\mathbb{S}^{n-1}$ is given by \thref{thm:vplaplacienFCT} and associated to:
\begin{itemize}
    \item exact $(n-1)$-forms for $k \ge 1$
    \item the volume form on $\mathbb{S}^{n-1}$ for $k = 0$.
\end{itemize}

For $k \ge 1$, for each Laplace eigenform $\xi_k \in \Omega^{n-1}(\mathbb{S}^{n-1})$, formula \eqref{eq:wrtheta1} provides an associated eigenform $\omega_k$ for the NBS such that $\iota^* \omega_k = \xi_k$ associated to the eigenvalue $\ell^{\Ex}_{k-1,n-1}$.\\ For $k = 0$, the reader can check using \propref{prop:conditions} and \rkref{rk:condition+} that taking $\xi_{\omega_2} = \xi_{\omega_1} = \mathrm{dVol}_{\mathbb{S}^{n-1}}$, $Q_{\omega_2} = r^n$, $Q_{\omega_1} = n r^n$, $P_{\omega_1} = P_{\omega_2} = 0$, $\eta_{\omega_1} = 0$ and $\eta_{\omega_2} = 0$, the formula \eqref{eq:w=W+H} and \eqref{eq:W=P+Q} provides an eigenform $\omega_0$ for NBS such that $\iota^* \omega_0 = \mathrm{dVol}_{\mathbb{S}^{n-1}}$, associated to the eigenvalue $\ell_0 = n^3(n+2)$. Again, since for each eigenform $\xi_k$ of $\bar{\Delta}$ on $\mathbb{S}^{n-1}$ we have an associated eigenform $\omega_k$ for NBS such that $\iota^* \omega_k = \xi_k$, the density argument of \propref{prop:spectrentier} shows that we have all the spectrum of NBS on $p$-forms.
\end{proof}
\section{Spectra on the ball and inequality testing}\label{sec:Test}
\subsection{Spectra of the Steklov and DBS operators on the Euclidean ball} As soon as we have a boundary, the spectrum is not completely known on the Euclidean ball. We recall in this section the Steklov and DBS spectra on differential forms. For Steklov eigenvalues on the ball, we follow the presentation of \cite{RaulotSavo3}. Let $P_{k,p}$ be the vector space of polynomial $p$-forms of homogeneous degree $k \ge 0$ on $\mathbb{R}^{n}$ and set:
\begin{equation*}
\begin{cases}
    H_{k,p} = \{\omega \in P_{k,p} : \Delta\omega = 0, \delta\omega = 0\}\\
    H'_{k,p} = \{\omega \in H_{k,p} : d\omega = 0\} \\
    H''_{k,p} = \{\omega \in H_{k,p} : r \frac{\partial}{\partial r} \lrcorner\omega = 0\}.
\end{cases}
\end{equation*} It turns out that $H_{k,p} = H'_{k,p} \oplus H''_{k,p}$ and $d : H''_{k,p} \rightarrow H'_{k-1,p+1}$ is a linear isomorphism for all $k \ge 1$. We set\begin{equation*}
    M_{k,p} = \dim(H''_{k,p})
\end{equation*}This number can be computed by representation theory, see \cite[Theorem 6.8]{AkiraYoshiharu}.

\begin{thm}\cite[Theorem 1]{RaulotSavo3}
The spectrum of the Dirichlet-to-Neumann operator on functions on the unit ball in $\mathbb{R}^n$ consists of the eigenvalues $\sigma^{\Coex}_{k,0} = k$, where $k = 0,1,2,\dots$, each with multiplicity $M_{k,0} = \dim(H_{k,0})$ and associated eigenspace $\iota^*(H_{k,0})$.
\end{thm}
\begin{thm}\cite[Theorem 2]{RaulotSavo3}\label{thm:spectreSteklov}
    Let $1 \le p \le n - 1$. The spectrum of the Dirichlet-to-Neumann operator on $p$-forms on the unit ball of $\mathbb{R}^n$ is given by the following two families of eigenvalues $\{\sigma^{\Ex}_{k,p}\}, \{\sigma^{\Coex}_{k,p}\}$ indexed by a positive integer $k = 1,2,\dots$:
\begin{equation*}
    \left\{
    \hspace{-0.3em}
    \begin{array}{r@{\quad}c@{\quad}l@{\quad}l}
        \sigma^{\Ex}_{k,p} &=& (k + p - 1) \dfrac{n + 2k }{n + 2k - 2} & \text{with multiplicity } M_{k,p-1} \\[0.5em]
        \sigma^{\Coex}_{k,p} &=& k + p & \text{with multiplicity } M_{k,p}
    \end{array}
    \right.
\end{equation*}
The eigenspace associated to $\sigma^{\Ex}_{k,p}$ is $\iota^*(H'_{k-1,p})$ and consists of exact forms. The eigenspace associated to $\sigma^{\Coex}_{k,p}$ is $\iota^*(H''_{k,p})$ and consists of co-exact forms. \end{thm} 

\bigskip
\noindent In degree $p = n-1$:

\medskip
\begin{thm}\cite[Theorem 3]{RaulotSavo3}
The spectrum of the Dirichlet-to-Neumann operator on the unit ball of $\mathbb{R}^{n}$ consists of the eigenvalues:
\begin{equation*}
    \sigma^{\Coex}_{1,n-1} = n
\end{equation*}
with multiplicity one (the associated eigenspace being spanned by the volume form of $\mathbb{S}^{n-1}$) and, for $k \ge 1$:
\begin{equation*}
    \sigma^{\Ex}_{k,n-1} = (k + n - 2) \frac{n + 2k}{n + 2k - 2}
\end{equation*}
with multiplicity $M_{k,n-2}$ and associated eigenspace $\iota^*H'_{k-1,n-1}$.\end{thm}

  For DBS eigenvalues on the ball, we follow the presentation of \cite{FerreroGazzolaWeth} and define for any $d \in \mathbb{N}$ the following space:
 \begin{equation*}
    \mathcal{D}_k := \{ P \in C^\infty(\mathbb{R}^n); \ \Delta P = 0 \text{ in } \mathbb{R}^n, \ P \text{ homogeneous polynomial of degree } k - 1 \}.
\end{equation*}
\begin{thm}\cite[Theorem 1.3]{FerreroGazzolaWeth}\label{thm:spectreDBS}
  On $\mathbb{B}^n$, for $n \ge 2$, we have for all $k = 1, 2, 3, \dots$:
\begin{itemize}
    \item[$\bullet$] the eigenvalues of \eqref{eq:DBSfct} are $q_k = n + 2(k - 1)$
    \item[$\bullet$] the multiplicity of $q_k$ equals $\dim(\mathcal{D}_k)$
    \item[$\bullet$] Every $q_k$-eigenfunction is of the form $(1 - |x|^2)\psi_k(x)$, where $\psi_k \in \mathcal{D}_k$.
\end{itemize}
\end{thm}
Actually, for $n=1$ the spectrum of DBS for functions satisfying \eqref{eq:DBSfct} on the interval $[-1,1]$ reduces to $\{1,3\}$, each with multiplicity $1$ (see \cite{BerchioGazzolaMitidieri}).

According to \cite[Remark 3.7]{BSFidaGeorgeOlaNicolas}, the spectrum without multiplicities of the DBS operator for scalar functions \eqref{eq:DBSfct} and for $p$-differential forms \eqref{eq:DBSF1} on the unit ball are the same. Therefore, we obtain identical eigenvalues but with different multiplicities. We have that $\operatorname{mult}(q_{k,p}) = \operatorname{mult}(q_k) \times \binom{n}{p}.$

\subsection{Testing a Kuttler-Sigillito inequality}\label{sec:KSTEST}
After determining the spectra of the NBS operator, we evaluate in this section inequality \eqref{eq:q1sigma1<l1PourLaPreuve} on the Euclidean unit ball. Given the spectrum of the Steklov operator \eqref{eq:steklovFormes} in \thref{thm:spectreSteklov}, the DBS \eqref{eq:DBSF1} in \thref{thm:spectreDBS}, and the NBS \eqref{eq:NBSF3preuve} in \thref{thm:spectreNBS}, we test how far the inequality \eqref{eq:q1sigma1<l1PourLaPreuve} is from being optimal on the ball.

\begin{thm}
For $n\geq3$ and $p\in\{1,\cdots,n-2\}$, we have the following relation on the Euclidean ball:
\begin{equation*}
\frac{q_{1,p} \sigma_{1,p}^2}{ \ell_{1,p}} = 
\begin{cases} 
\dfrac{n^2+4n-4p}{n(n+4)} & \text{if } 1\leq p \leq \frac{n}{2} \\[2ex]
\dfrac{n}{n+2} & \text{if } p > \frac{n}{2}.
\end{cases}
\end{equation*} 
\end{thm} Therefore, we recover $q_{1,p}\sigma_{1,p}^2<\ell_{1,p}.$

\begin{proof}
For $k=1,$ we obtain from \thref{thm:spectreSteklov} that $\sigma^{\Coex}_{1,p}=p+1$ and $\sigma^{\Ex}_{1,p}=p(1+\frac{2}{n}).$ We have that $\sigma^{\Ex}_{1,p}\leq\sigma^{\Coex}_{1,p}$ if and only if $p\leq\frac{n}{2}.$ Consequently, \[
    \sigma_{1,p} = 
    \left\{
    \hspace{-0.3em}
    \begin{array}{r@{\quad}c@{\quad}l@{\quad}l}
        \sigma^{\Ex}_{1,p} &=& p\left(1 + \dfrac{2}{n}\right) & \text{if } p \le \dfrac{n}{2} \\[0.6em]
        \sigma^{\Coex}_{1,p} &=& p + 1 & \text{if } p > \dfrac{n}{2}.
    \end{array}
    \right.
\]We also have by \thref{thm:spectreDBS} that $q_{1,p}=n.$ From \thref{thm:spectreNBS}, we obtain$$\ell^{\Coex}_{0,p}=(p+1)^2(n+2) \quad \text{and} \quad \ell^{\Ex}_{0,p}=\frac{p^2(n+2)^2(n+4)}{n^2+4n-4p}.$$ Therefore, we have the following:
$$\ell_{0,p} = 
\left\{
\hspace{-0.3em}
\begin{array}{r@{\quad}c@{\quad}l@{\quad}l}
    \ell^{\Ex}_{0,p} &=& \dfrac{p^2(n+2)^2(n+4)}{n^2+4n-4p} & \text{if } p \le \dfrac{n}{2} \\[1.2em]
    \ell^{\Coex}_{0,p} &=& (p+1)^2(n+2) & \text{if } p > \dfrac{n}{2}.
\end{array}
\right.
$$

For $1<p \le \frac{n}{2}$, we verify the inequality:
$$n p^2 \left(1 + \frac{2}{n}\right)^2 <\frac{p^2(n+2)^2(n+4)}{n^2+4n-4p}.$$ We have $$\frac{q_{1,p} \sigma_{1,p}^2}{ \ell_{1,p}} = \dfrac{n^2+4n-4p}{n(n+4)} <1$$ therefore the strict inequality holds.

For $p > \frac{n}{2}$, we have:
$$n(p+1)^2 < (p+1)^2(n+2)$$ and $$\frac{q_{1,p} \sigma_{1,p}^2}{ \ell_{1,p}}= \dfrac{n}{n+2}<1.$$ \end{proof}

\bibliographystyle{abbrv}
\bibliography{references}

\end{document}